\documentclass{article}

\usepackage[T1]{fontenc}
\usepackage{amsmath}
\usepackage{amssymb}
\usepackage{amsthm}
\usepackage{mathtools}
\usepackage{url}
\usepackage{graphicx}
\usepackage{subcaption}
\usepackage{bigfoot} 
\usepackage{filecontents}

\usepackage{geometry}

\usepackage{hyperref}
\hypersetup{
	pdftoolbar=true,
	pdfmenubar=true,
	pdffitwindow=false,
	pdfstartview={FitH},
	pdftitle={},
	pdfauthor={Dario Corona},
	pdfsubject={},
	pdfkeywords={},
	pdfnewwindow=true,
	colorlinks=true, 
	linkcolor=blue,
	citecolor=blue,
	urlcolor=blue,
}

\DeclarePairedDelimiter{\norm}{\lVert}{\rVert}

\DeclareMathOperator\supp{supp}

\theoremstyle{plain}\newtheorem{theorem}{Theorem}
\theoremstyle{plain}\newtheorem{proposition}[theorem]{Proposition}
\theoremstyle{plain}\newtheorem{lemma}[theorem]{Lemma}
\theoremstyle{plain}

\theoremstyle{definition}
\theoremstyle{remark}\newtheorem{remark}[theorem]{Remark}
\theoremstyle{definition}\newtheorem{example}[theorem]{Example}

\newtheorem*{duboisreymond}{DuBois-Reymond Lemma}

\numberwithin{theorem}{section} 
\numberwithin{equation}{section}

\newcommand{\covD}{\mathrm{D}}
\newcommand{\covI}{\int\!\!\!\!\!\!\int}

\mathtoolsset{showonlyrefs}

\title{
A note on the regularity and the existence of Riemannian splines}
\author{D. Corona, R. Giamb\`o, P. Piccione}

\begin{document}
\maketitle

\begin{abstract}
	In this paper,
	we present a comprehensive proof concerning
	the regularity of critical points
	for the spline energy functional on Riemannian manifolds,
	even for the general higher-order case.
	Although this result is widely acknowledged
	in the literature, a detailed proof was previously absent.
	Our proof relies on a
	generalization of the DuBois-Reymond Lemma.
	Furthermore, we establish the existence of minimizers
	for the spline energy functional in cases
	where multiple interpolation points are prescribed
	alongside just one velocity.
\end{abstract}



\section{Introduction}

This paper analyzes the regularity and existence of spline
curves on a Riemannian manifold $(M,g)$,
defined variationally as the critical points of the 
\emph{spline energy functional}
$f\colon \Gamma \to M$
that reads as follows:
\begin{equation}
	\label{eq:def-f-intro}
	f(\gamma) \coloneqq
	\frac{1}{2}
	\int_{0}^{1}
	g\big(\covD_t \dot{\gamma}(t),\covD_t \dot{\gamma}(t)\big)
	\mathrm{d}t,
\end{equation}
where
$\Gamma$ is a suitable subspace of the Sobolev functions
$W^{2,2}([0,1],M) = H^2([0,1],M)$
and 
$\covD_t$ denotes the usual covariant derivative
given by $g$
(more formally, by its Levi-Civita connection).
Usually, $\Gamma$ is the set of curves that
pass through some prescribed points of the manifolds at 
prescribed time instants (interpolation conditions),
with some fixed velocities on some of these points.
When $M$ is an euclidean space $\mathbb{R}^{n}$,
the critical points of such a functional are the cubic splines,
namely piecewise third degree polynomials,
from which this functional takes its name.

This problem has been introduced in the Riemannian setting 
by L. Noakes, G. Heinzinger and B. Paden
in~\cite{Noakes1989},
where $\Gamma$ is the set of \emph{smooth} curves 
with prescribed initial and final points and velocities.
Under this regularity assumption
on the curves one considers, 
~\cite[Theorem 1]{Noakes1989}
states that
if $\gamma$ is a critical point of $f$ then
\begin{equation}
	\label{eq:Riemannian-spline-intro}
	\covD^3_t \dot{\gamma}(t) + R(\covD \dot\gamma(t),\dot{\gamma}(t))\dot{\gamma}(t)
	= 0,
	\quad\text{for any } t \in [0,1],
\end{equation}
where $R$ denotes the Riemannian curvature tensor
of the metric $g$ and $\covD^3_t$
denotes the third order covariant derivative.
The fourth-order differential
equation~\eqref{eq:Riemannian-spline-intro}
is derived from the critical point condition.
This is done by starting from the first variation formula of $f$,
performing explicit computations,
using the integration by parts formula,
and then applying the fundamental lemma of
the calculus of variations,
which allows to obtain a differential
equation from an integral one.

However, if one assume a priori that the critical point $\gamma$
has only the $H^2$--regularity,
which is the ``natural'' regularity coming from the definition of $f$,
the integration by parts is not allowed,
and the above procedure prematurely stops.
Indeed, by standard computations,
the first variation of $f$ at $\gamma$
in the direction $\xi$ is as follows: 
\begin{equation}
	\label{eq:first-variation-intro}
	\mathrm{d}f(\gamma)[\xi]
	= \int_{0}^{1}
	\Big(
		g\big(\covD_t\dot{\gamma},\covD^2_t\,\xi\big)
		- g\big(R\big(\dot{\gamma},\covD_t\dot{\gamma}\big)\dot{\gamma},\xi\big)
	\Big) \mathrm{d}t,
\end{equation}
where, at this point in the presentation, one can assume
that $\xi\colon [0,1] \to TM$
is a smooth vector field on $\gamma$
with compact support in $(0,1)$.
If $\gamma$ is only of class $H^2$,
then $\covD_t \dot{\gamma}$ is a
vector field on $\gamma$ with $L^2$ regularity.
Since the minimal regularity assumption to perform
the integration by parts is the existence of a weak derivative,
the following step,
\[
	\int_{0}^1	
	g\big(\covD_t\dot{\gamma},\covD^2_t\,\xi\big)
	\mathrm{d}t
	= 
	\int_{0}^1	
	g\big(\covD_t^3\dot{\gamma},\xi\big)\mathrm{d}t,
\]
cannot be executed
since $\covD_t^2\dot{\gamma}$ does not exist for a curve of class $H^2$.

While the literature on Riemannian splines is extensive
due to their numerous applications in control theory,
robotics, computational graphics,
geometric mechanics, to name a few,
it appears that the regularity problem is often circumvented
by imposing certain smoothness conditions
on the curves under consideration.
For example, the pioneering work~\cite{Noakes1989}
considers only smooth functions,
while~\cite{Crouch1991},
where multiple interpolation points are considered,
assumes to work with $C^1$ piecewise smooth curves,
as done also in~\cite{Camarinha2001}.
Jumping to more recent works, 
~\cite{Heeren2019} provides a regularity result
for the minimizers of a slightly different functional
(the one that weights the spline energy functional
with the path energy functional).
Starting from the $H^2$--regularity,
~\cite[Theorem 2.19]{Heeren2019} states
that the minimizers are
in $C^{2,\frac{1}{2}}([\delta,1-\delta],M)$,
for any $\delta \in (0,1/2)$.
However, this is still not sufficient to perform 
an integration by parts
and, consequently, to obtain~\eqref{eq:Riemannian-spline-intro}.
Indeed, if a minimizer $\gamma$ is of class $C^{2,\frac{1}{2}}$,
then its second order derivative
is a vector field of class $C^{0,\frac{1}{2}}$
which can not admit a weak derivative.
In the one dimensional Riemannian setting, a typical function
with the above properties is the well-known Weierstrass function,
which is H\"older continuous but it is not absolutely continuous.
If one requires the $C^{0,\frac{1}{2}}$--regularity,
then the real parameters $a,b$ in the Weierstrass function 
\[
	W(t) = \sum_{k  = 0}^{\infty}
	a^k\cos(b^k\pi t),
\]
have to be chosen such that $b = 1/a^2 \in \mathbb{N}$,
$a \in (0,1)$ and $1/a > 1 + \frac{3}{2}\pi$.
With these parameters,
$W \in C^{0,\frac{1}{2}}(\mathbb{R},\mathbb{R})$
but, as it doesn't admit a derivative at any point,
it is not absolutely continuous.

The study of the splines on Riemannian manifolds
has also application in the collision avoidance of multiple agents,
as investigated, for instance,
in~\cite{Assif2018,Colombo2020,goodman2022},
where the spline energy functional
is combined both with the path energy one
and with a distance function between multiple curves.
Once again, in~\cite{Assif2018,goodman2022}
the regularity issue is not directly tackled,
since the functional space is given
by the $C^1$ piecewise smooth curves.
In~\cite{Colombo2020}, the functional space is
given among the curves with $H^2$--regularity
and it is stated that the critical points are smooth;
however, the proof is omitted, mentioning the results
in~\cite{Assif2018}, where the piecewise smoothness is assumed.

In recent years, the spline problem has also
garnered interest for the
challenging and intriguing study of its flow,
known as \emph{elastic flow},
which involves the evolution of curves
following the negative gradient of the
spline energy functional,
leading to the analysis of a fourth-order parabolic quasilinear PDE.
The study of this flow began in~\cite{polden1996},
but for a detailed survey and references,
we refer to~\cite{Pozzetta2021}
and the works cited therein.
Even within this framework, the admissible curves
exhibit more than $H^2$-regularity.
For instance,
both in~\cite{Pozzetta2021} and in~\cite{Pozzetta2022},
they are of class $H^4$.
A similar analysis was recently carried out in~\cite{lin2024},
where the elastic flow is studied in a broader context.
Even in this scenario,
when constrained to the original spline energy functional,
the considered curves are of class $C^2$ and piecewise
$C^{4,\alpha}$, with $\alpha \in (0,1)$.

Finally, let us observe that in~\cite{GGP2002} and~\cite{GGP2004},
the regularity of the critical curves is not assumed a priori;
however, its proof is given, without further details,
invoking a ``standard'' bootstrap method.
The bootstrap method, usually used to prove the regularity 
of critical points of a Lagrangian action functional,
requires integration by parts in the ``wrong'' direction,
namely the only one which is possible under the regularity 
assumed at the beginning.
For example, starting from~\eqref{eq:first-variation-intro},
one should obtain the following:
\begin{equation}
	\label{eq:first-var-integratedtwice}
	\mathrm{d}f(\gamma)([\xi])
	= \int_{0}^{1}
	\Big(
		g\big(\covD_t\dot{\gamma},\covD^2_t\,\xi\big)
		- g\big(\mu(t),\covD^2_t\xi\big)
	\Big) \mathrm{d}t,
\end{equation}
where $\mu(t)$ is a vector field on $\gamma$
whose second order covariant derivative 
is $R(\dot{\gamma},\covD_t\dot{\gamma})\dot{\gamma}$,
hence it is of class $H^2$,
and $\mu(0) = 0$, $\covD_t \mu(0) = 0$.
It then applies the DuBois-Reymond Lemma, which,
for the reader's convenience, is reported here
in the simpler setting of real functions
(see, e.g., \cite[Corollary 1.25]{Dacorogna-Intro}).
\begin{duboisreymond}
	Let $u \in L^1_{loc}((a,b),\mathbb{R})$
	such that
	\[
		\int_{a}^b u(t) \eta(t)\mathrm{d}t = 0,
		\qquad \forall \eta \in C^{\infty}_c([a,b],\mathbb{R})
		\quad\text{s.t. } \int_a^b \eta(t)\mathrm{d}t = 0.
	\]
	Then, there exists $c \in \mathbb{R}$ such that
	$u = c$ a.e. in $(a,b)$.
\end{duboisreymond}
The main problem to apply the ``standard'' bootstrap method is as follows.
From the critical point condition $\mathrm{d}f(\gamma) = 0$,
hence from the null condition on~\eqref{eq:first-var-integratedtwice},
one can't infer that $\covD_t \dot{\gamma}(t) = v + \mu(t)$ a.e.  on $[0,1]$
for some \emph{parallel} vector $v$ field on $\gamma$
(i.e., $\covD_t v \equiv 0$),
obtaining that $\covD_t \dot{\gamma}$ is of class $H^2$
(so that one can infer
the spline equation~\eqref{eq:Riemannian-spline-intro}
as previously described).
This is not possible because the set of possible variations 
one can apply to $\gamma$ 
it is strictly contained within the set of the variations
required to apply the DuBois-Reymond Lemma.
Indeed, if $\xi$ has compact support in $(0,1)$,
not only 
$ \int_{0}^{1} \covD^2_t \xi(t) \mathrm{d}t = 0$,
which is obtained from the condition
that the covariant derivative
of $\xi$ is $0$ at the extreme points,
but we have the additional constraints that 
$\xi$ has to be null at the extremes.
As a consequence, the ``standard'' bootstrap method
mentioned in~\cite{GGP2002,GGP2004} does not work,
or rather, it requires some further clarification.

In this paper, we demonstrate
that a bootstrap method can be indeed 
successfully applied
by utilizing a generalization of the DuBois-Reymond Lemma,
namely Lemma~\ref{lem:affine-ortho},
thereby proving the regularity of critical points.
More formally, we have the following result.

\begin{theorem}
	\label{theorem:regularity}
	Let $(M^n,g)$ be an $n$--dimensional smooth
	and complete Riemannian manifold.
	Let $\{p_0, \ldots, p_N\} \subset M$ be a set of
	$N+1$ fixed points,
	$0 = t_0 < t_1 < \dots < t_N = 1$
	a partition of the unit interval,
	and let us fix an index $j \in \{0, \ldots, N\}$,
	together with a vector $v \in T_{p_j}M$.
	Let 
	\[
		\Gamma = \left\{
			\gamma \in H^2([0,1],M):
			\gamma(t_i) = p_i,
			\, \forall i = 0, \ldots, N,
			\, 
			\dot{\gamma}(t_j) = v
		\right\},
	\]
	and $f \colon \Gamma \to \mathbb{R}$ be defined by~\eqref{eq:def-f-intro}.
	Then, if $\gamma \in \Gamma$ is a critical point of $f$,
	it is smooth on any interval $[t_{i-1}, t_i]$.
	As a consequence, for any $i = 1, \ldots, N$
	we have
	\begin{equation}
		\label{eq:Riemannian-spline}
		\covD^3_t \dot{\gamma}(t)
		+ R\big(\covD_t \dot{\gamma}(t), \dot{\gamma}(t)\big)
		\dot{\gamma}(t)
		= 0,
		\quad \text{for any } t \in [t_{i-1}, t_i].
	\end{equation}
	Moreover, both $\gamma|_{[0, t_j]}$
	and $\gamma|_{[t_j, 1]}$ are of class $C^2$;
	if $t_j \ne 0$, then $\covD_t \gamma(0) = 0$,
	and if $t_j \ne 1$, then $\covD_t \gamma(1) = 0$.
\end{theorem}

\begin{remark}
	The regularity results remain analogous even
	when velocities are prescribed at multiple points.
	In such cases, a critical point $\gamma$ will be of class
	$C^2$ within any interval delimited
	by two instants where velocity is prescribed.
	The proof for this follows the same argument
	as that used in Theorem~\ref{theorem:regularity},
	with the crucial observation that
	each point where velocity is prescribed
	must be treated similarly to the only one
	in Theorem~\ref{theorem:regularity}.
	To maintain clarity and avoid heavy notation
	that might obscure the main proof steps,
	we focus on using only one point with prescribed velocity.
	If no velocity is prescribed, and a critical point exists,
	then it is of class $C^2$ on the whole unit interval,
	and the covariant derivative of its velocity vanishes
	at the extrema.
\end{remark}

\begin{remark}
	The above regularity theorem holds even if 
	$M$ is a differentiable manifold of class $C^4$
	and the metric tensor $g$ is of class $C^3$,
	as these are the minimal (reasonable) assumptions
	so that~\eqref{eq:Riemannian-spline} can be explicitly written
	in local charts.
\end{remark}

The regularity ensured by Theorem~\ref{theorem:regularity}
can be generalized for higher-order Riemannian 
splines, the so-called \emph{$k$-splines},
which are, for any integer $k \ge 2$,
the critical points of the energy functional
$f_k\colon \Gamma\subset H^{k}([0,1],M) \to \mathbb{R}$
given by
\begin{equation}
	\label{eq:f-k-Riemann}
	f_k(\gamma)\coloneqq
	\int_0^1 g\big(\covD_t^{k-1}\dot{\gamma},
	\covD_t^{k-1}\dot{\gamma}\big)
	\mathrm{d}t.
\end{equation}
Once again, if the regularity result is ensured, 
then applying the integration by parts formula
the necessary number of times
and the fundamental lemma of the calculus of variations
leads to an ODE of order $2k$
on any interval $I \subset [0,1]$
delimited by some interpolation conditions.
However, even in this case, we couldn't find 
satisfactory regularity results for the critical points
of the functional~\eqref{eq:f-k-Riemann},
as usually one works in the set 
of piecewise smooth curves or,
at least, assuming the 
$C^{2k-2}$ regularity on each interval
of the admissible curves
(see, e.g.,~\cite{Machado2010,lin2024}).
Therefore, we provide the following result,
which is based on the further generalization of the 
DuBois-Reymond Lemma
presented in Lemma~\ref{lem:DuBois-k-Riemannian}.

\begin{theorem}
	\label{theorem:regularity-k-Riemann}
	Let $(M^n,g)$ be an $n$--dimensional smooth
	and complete Riemannian manifold.
	Let $\{p_0, \ldots, p_N\} \subset M$ be a set of
	$N+1$ fixed points,
	$0 = t_0 < t_1 < \dots < t_N = 1$
	a partition of the unit interval,
	and let us fix an index $j \in \{0, \ldots, N\}$,
	together with $k-1$ vectors $v_\ell \in T_{p_j}M$,
	for $\ell = 1,\dots,k-1$.
	Let 
	\[
		\Gamma_k = \Big\{
			\gamma \in H^k([0,1],M):
			\gamma(t_i) = p_i,
			\, \forall i = 0, \ldots, N,
			\, 
			\covD_t^{\ell-1}\dot{\gamma}(t_j) = v_{\ell},
			\, \forall \ell = 1,\dots,k-1
		\Big\},
	\]
	and $f_k \colon \Gamma_k \to \mathbb{R}$ be defined 
	by~\eqref{eq:f-k-Riemann}.
	Then, if $\gamma \in \Gamma$ is a critical point of $f_k$,
	it is smooth on any interval $[t_{i-1}, t_i]$.
	As a consequence, for any $i = 1, \ldots, N$
	we have
	\begin{equation}
		\label{eq:Riemannian-k-spline}
		\covD_t^{2k - 1}\dot{\gamma}(t)
		+
		\sum_{\ell = 0}^{k-2}
		(-1)^{\ell}
		R\big(\covD_t^{2k - \ell - 3}\dot{\gamma}(t),
		\covD_t^{\ell}\dot{\gamma}(t)
		\big)\dot{\gamma}(t)
		= 0,
		\quad \text{for any } t \in (t_{i-1}, t_i).
	\end{equation}
	Moreover, both $\gamma|_{[0, t_j]}$
	and $\gamma|_{[t_j, 1]}$ are of class $C^{2k-2}$;
	if $t_j \ne 0$, then $\covD_t^{k + \ell-2} \dot{\gamma}(0) = 0$,
	for $\ell= 1,\dots,k-1$,
	and similarly for $t_j \ne 1$.
\end{theorem}

\begin{remark}
	\label{rem:just-one-j}
	Theorem~\ref{theorem:regularity-k-Riemann} 
	can also be stated and proved
	in a slightly more general setting,
	by fixing the covariant derivatives of the curves
	at different points and by using different orders,
	where the latter are obviously always between $1$ and $k-1$.
	More formally, for each $\ell = 1,\dots,k-1$,
	one can choose an index $j_\ell \in \{0,\dots,N\}$
	and an order of derivative $w_\ell \in \{0,\dots,k-2\}$,
	and define
	\[
		\Gamma_k =
		\Big\{ \gamma \in H^k([0,1],M):
			\gamma(t_i) = p_i, \, \forall i = 0, \ldots, N,
			\covD_t^{w_\ell}\dot{\gamma}(t_{j_\ell}) = v_\ell,
		\, \forall \ell = 1,\dots,k-1 \Big\},
	\]
	where each $v_\ell$ belongs to $T_{p_{j_\ell}}M$.
	In this setting,~\eqref{eq:Riemannian-k-spline} still holds,
	but the last regularity results stated 
	in the theorem should be modified accordingly.
	However, for the sake of clarity and simplicity,
	we confine ourselves to the setting given in the theorem.
\end{remark}

\bigskip

Beside a comprehensive proof of
the regularity of critical points
for the spline energy functional, 
this paper contains an existence result
for the minimizers of the same functional.
The existence result we give
stems from the analysis of \cite[Lemma 2.15]{Heeren2019},
where it is proven that, in general, a minimizer of the 
spline energy functional can not exist if only the 
interpolation points are prescribed.
For the reader's convenience, the counterexample
provided as a proof of~\cite[Lemma 2.15]{Heeren2019}
is reported and discussed in details in
Example~\ref{ex:cylinderExample}.
Based on that remarkable example,
~\cite{Heeren2019} justifies the introduction of additional
constraints to obtain the existence of minimizers,
which are as follows:
\begin{itemize}
	\item \emph{natural} boundary conditions, i.e., 
		$ \covD_t \dot{\gamma}(0) = \covD_t \dot{\gamma}(1) = 0$;
	\item \emph{periodic} boundary conditions, i.e.,
		$p_0 = p_N$ and 
		$ \dot{\gamma}(0) = \dot{\gamma}(1)$;
	\item \emph{Hermite} boundary conditions, i.e.,
		the prescription of
		\emph{the initial and final velocities} of the curves.
\end{itemize}
Additionally,~\cite{Heeren2019} weights 
the spline energy functional with the path energy,
namely looking for the minimizers of 
\[
	F^{\sigma}(\gamma)\coloneqq
	f(\gamma)
	+ 
	\sigma
	\int_{0}^1 g\left(\dot{\gamma},\dot{\gamma}\right)
	\mathrm{d}t,
\]
for some $\sigma > 0$.
Indeed, in~\cite[Theorem 2.19]{Heeren2019},
the existence of a minimizer for $F^{\sigma}$ is guaranteed if $\sigma > 0$
and if one of the above constraints is assumed.

A direct inspection of the construction 
provided to show the non-existence of a minimizer
in~\cite[Lemma 2.15]{Heeren2019}
(see also Example~\ref{ex:cylinderExample})
suggests that, if \emph{just one} velocity
is prescribed, then a minimizer should exists,
even when $\sigma = 0$.
More formally, we have the following theorem,
whose setting is depicted in Figure~\ref{fig:manifold}.
\begin{theorem}
	\label{theorem:existence}
	Let $(M^n,g)$ be an $n$--dimensional
	complete Riemannian manifold.
	Let $\{p_0,\ldots,p_N\}\subset M$ be a set
	of $N+1$ fixed points
	and $0 = t_0 < t_1 < \dots < t_N = 1$
	a partition of the unit interval.
	Let us fix a vector $v \in T_{p_0}M$
	and let 
	\[
		\Gamma = \Big\{
			\gamma \in H^2([0,1],M):
			\gamma(t_i) = p_i,
			\, \forall i= 0,\dots, N,
			\,
			\dot{\gamma}(0) = v
		\Big\}.
	\]
	Then, the functional $f\colon \Gamma \to \mathbb{R}$ defined 
	by~\eqref{eq:def-f-intro} admits a minimizer.
\end{theorem}

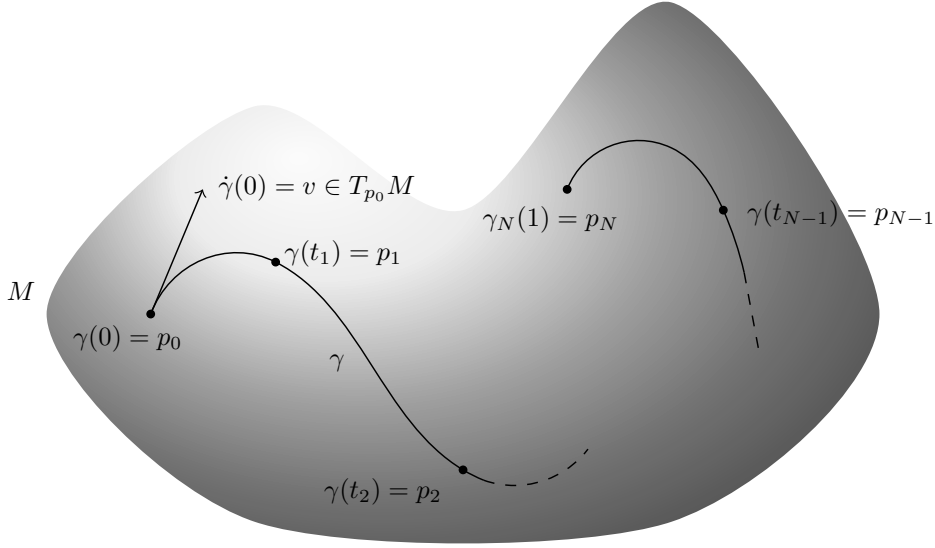
\begin{figure}[t]
	\centering
	\begin{picture}(315,213)
		\put(0,0){\includegraphics[width=315pt,height=213pt]{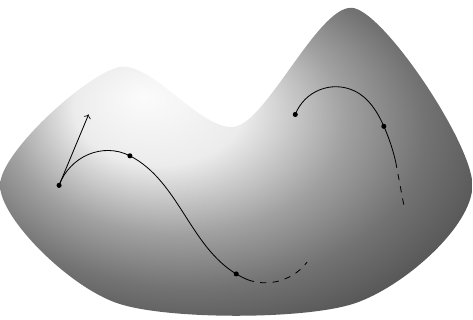}}
		\put(-15,95){$M$} 
		\put(10,78){$\gamma(0) = p_0$} 
		\put(90,110){$\gamma(t_1) = p_1$} 
		\put(105,20){$\gamma(t_2) =  p_2$} 
		\put(265,125){$\gamma(t_{N-1}) = p_{N-1}$} 
		\put(165,123){$\gamma_N(1) = p_N$} 
		\put(65,135){$\dot{\gamma}(0) = v \in T_{p_0}M$} 
		\put(107,70){$\gamma$} 
	\end{picture}
	\caption{The setting of Theorem~\ref{theorem:existence}:
		 on a Riemannian manifold $M$, 
		 we prescribe $N + 1$ points,
		 denoted by $p_0,\dots,p_N$,
		 together with the initial velocity $v \in T_{p_0}M$.
	 }
    \label{fig:manifold}
\end{figure}

\begin{remark}
	Prescribing the velocity at the initial time 
	doesn't affect the generality of the theorem.
	Indeed, if one prescribes the velocity at an 
	internal time instant $t_j \ne \{0,1\}$, 
	then Theorem~\ref{theorem:existence} ensures the 
	existence of two minimizers of $f$ 
	in the two time intervals $[0,t_j]$ 
	and $[t_j,1]$, 
	say $\gamma_1\colon [0,t_j] \to M$ 
	and $\gamma_2\colon [t_j,1] \to M$.
	For the sake of precision, this follows 
	also from the invariance of the spline energy functional
	with respect to backward reparametrization
	of the curves.
	By Theorem~\ref{theorem:regularity},
	the curve $\gamma\colon [0,1] \to M$ 
	obtained by gluing together $\gamma_1$ and $\gamma_2$ 
	is of class $H^2$,
	as its second order derivative could have
	just one step discontinuity at time $t_j$.
	Therefore, $\gamma$ is a minimizer of $f$.
\end{remark}

The rest of the paper is organized as follows.
Section~\ref{sec:onedimensionalregularity}
presents a proof of Theorem~\ref{theorem:regularity}
in the one-dimensional case,
allowing the key steps of the proof
to be appreciated without excessive notation.
Section~\ref{sec:Riemannian-regularity} then provides 
a detailed proof of Theorem~\ref{theorem:regularity}, 
encompassing all necessary details for the Riemannian setting.
Section~\ref{sec:k-regularity} is dedicated
to the proof of Theorem~\ref{theorem:regularity-k-Riemann},
beginning again with the analogous result
in the one dimensional setting.
Finally, Section~\ref{sec:existence} focuses on the proof of 
Theorem~\ref{theorem:existence}.

\section{One dimensional setting}
\label{sec:onedimensionalregularity}
In this section, we provide the regularity result for 
the critical points of the spline energy functional for 
one dimensional curves,
namely maps defined on $[0,1]$
with values in $\mathbb{R}$.
Let $N + 1 \ge 2$ be the number of fixed points,
which we denote by $p_0, p_1, \ldots, p_N \in \mathbb{R}$,
and let us choose a partition of the unit interval
$ 0 = t_0 < t_1 < \ldots < t_N = 1$.
We impose that every curve we consider should pass
through the point $p_i$ at the time $t_i$,
for every $i = 0, \ldots, N$,
and we prescribe the velocity of the curves
at just \emph{one} point,
i.e., there exists $j \in \{0, \ldots, N\}$ and $v \in \mathbb{R}$
such that every curve we consider has a derivative equal to $v$
at the time $t_j$.
More formally, our functional space is defined as follows:
\begin{equation}
	\label{eq:def-X}
	\Gamma \coloneqq \Big\{\gamma \in H^2([0,1], \mathbb{R}) : 
		\gamma(t_i) = p_i, \, \forall i = 0, \ldots, N,
	\, \dot{\gamma}(t_j) = v \Big\},
\end{equation}
where $H^2([0,1],\mathbb{R})$ denotes the space
of Sobolev real functions defined on $[0,1]$
that admits a second order weak derivative 
which belongs to the Lebesgue space $L^2([0,1],\mathbb{R})$.
The main objective of this section is 
studying the regularity of
the critical points of the spline energy functional
$f\colon \Gamma \to \mathbb{R}$, which in this case reads 
simply as follows:
\begin{equation}
	\label{eq:def-f}
	f(\gamma) =
	\frac{1}{2}\int_{0}^{1} \ddot{\gamma}^2(t) \, \mathrm{d}t.
\end{equation}

By some standard computations, the set of admissible variations is given by
\begin{equation}
	\label{eq:def-V}
	V \coloneqq \left\{ 
		\xi \in H^2([0,1], \mathbb{R}) :
		\xi(t_i) = 0, \, \forall i = 0, \ldots, N, \,
		\dot{\xi}(t_j) = 0
	\right\},
\end{equation}
so that, for every $\gamma \in \Gamma$, the differential of $f$ at $\gamma$,
that we denote by $\mathrm{d}f(\gamma)\colon V \to \mathbb{R}$,
is
\begin{equation}
	\label{eq:diff-f}
	\mathrm{d}f(\gamma)[\xi]
	= \int_{0}^{1} \ddot{\gamma}(t) \, \ddot{\xi}(t) \, \mathrm{d}t.
\end{equation}
\begin{remark}
	It is important to notice that if $\gamma$ is a critical point of $f$,
	we cannot apply the DuBois-Raymond lemma to deduce its regularity.
	Let us proceed with the computation to better highlight this statement,
	which is a key observation for all the subsequent discussion
	about the regularity of the critical points.
	Assume that $\gamma$ is a critical point,
	so that $\mathrm{d}f(\gamma)[\xi] = 0$ for every $\xi \in V$.
	Let us fix $i = 1,\dots, N$
	and restrict our analysis on the variations
	with compact support in $(t_{i-1}, t_{i})$,
	so we have
	\[
		\mathrm{d}f(\gamma)[\xi] =
		\int_{t_{i-1}}^{t_i} \ddot{\gamma}(t) \, \ddot{\xi}(t) \, \mathrm{d}t = 0,
		\quad \forall \xi \in V, \, \supp\xi \subset (t_{i-1},t_{i}).
	\]
	Notice that, if $\supp\xi \subset (t_{i-1},t_{i})$,
	then $\dot{\xi}(t_{i-1}) = \dot{\xi} \, (t_{i}) = 0$,
	so that $\int_{0}^{1} \ddot{\xi} \, \mathrm{d}t = 0$.
	By the DuBois-Raymond lemma,
	if we can consider \emph{all} the variations $\xi$ such that
	$\int_{0}^1 \ddot{\xi} \, \mathrm{d}t = 0$,
	then we can conclude that $\ddot{\gamma} = c$ a.e. on $(t_{i-1},t_{i})$,
	hence we obtain a regularity result.
	However, we have to take into account also the condition
	$\xi(t_{i-1}) = \xi(t_{i}) = 0$,
	so \emph{not all} the variations required for the application
	of the DuBois-Raymond lemma are available.
\end{remark}

Building on the previous remark,
we need a characterization result about the admissible variations.
\begin{lemma}
	\label{lem:characterize-V}
	For every $\xi \in V$ 
	and $i = 1,\dots,N$,
	we have
	\begin{equation}
		\label{eq:mediaNulla}
		\int_{t_{i-1}}^{t_i}
		\dot{\xi}\, \mathrm{d}t
		= 0,
	\end{equation}
	and
	\begin{equation}
		\label{eq:charV}
		t_{i}\dot{\xi}(t_{i})
		- t_{i-1}\dot{\xi}(t_{i-1})
		= \int_{t_{i-1}}^{t_i}t\,\ddot{\xi}\mathrm{d}t.
	\end{equation}
\end{lemma}
\begin{proof}
	The condition~\eqref{eq:mediaNulla}
	naturally arises from the hypotheses that
	$\xi(t_i) = 0$ for each $i = 0, \dots, N$.
	Indeed, since $\xi(t_0) = 0$, 
	we have
	\[
		\xi(t_1) = \xi(t_0) + \int_{t_0}^{t_1}\dot{\xi}\,\mathrm{d}t
		= \int_{t_0}^{t_1}\dot{\xi}\,\mathrm{d}t = 0,
	\]
	and similarly for all subsequent indices.   
	Then,~\eqref{eq:charV} follows directly from~\eqref{eq:mediaNulla}
	and the integration by parts formula as follows:
	\[
		\int_{t_{i-1}}^{t_i}t\,\ddot{\xi}(t)\mathrm{d}t
		= 
		t_{i}\dot{\xi}(t_{i})
		- t_{i-1}\dot{\xi}(t_{i-1})
		- \int_{t_{i-1}}^{t_i}
		\dot{\xi}\,\mathrm{d}t
		=
		t_{i}\dot{\xi}(t_{i})
		- t_{i-1}\dot{\xi}(t_{i-1}).
	\]

\end{proof}

To state the next result, we need some further notation.
Let us define
\[
	V_0 \coloneqq \left\{
		\xi \in V: \dot{\xi}(t_i) = 0, \,
		\forall i=0,\dots,N
	\right\} \subset  V.
\]
By Lemma~\ref{lem:characterize-V},
if $\xi \in V_0$ then
for every $i=0,\dots,N-1$ we have
\[
	\int_{t_{i-1}}^{t_i}
	\ddot{\xi}(t)\,\mathrm{d}t = 0
	\quad\text{and}\quad
	\int_{t_{i-1}}^{t_i}t\,\ddot{\xi}(t)\,\mathrm{d}t = 0,
\]
and viceversa, since $\dot{\xi}(t_j) = 0$.
In other words, 
setting
\[
	W_0\big([a,b],\mathbb{R}\big) \coloneqq
	\left\{
		\eta \in L^2([a,b],\mathbb{R}):
		\int_{a}^{b}
		\eta(t)\,\mathrm{d}t = 0
		\text{ and }
		\int_{a}^{b}t\,\eta(t)\,\mathrm{d}t = 0
	\right\},
\]
for every interval $[a,b]\subset \mathbb{R}$,
we have
\begin{equation}
	\label{eq:V_0-otherChar}
	V_0 =
	\left\{
		\xi \in V:
		\ddot{\xi}\big|_{[t_{i-1},t_{i}]} \in W_0([t_{i-1},t_{i}],\mathbb{R})
		\text{ for every }i = 1,\dots,N
	\right\}.
\end{equation}

Now, we are ready to state the following result,
which generalizes the DuBois-Reymond lemma.
\begin{lemma}
	\label{lem:affine-ortho}
	If $u \in L^2([a,b],\mathbb{R})$
	is such that 
	\[
		\int_{a}^{b} u(t)\eta(t) \mathrm{d}t = 0,
		\qquad \forall \eta \in W_0([a,b],\mathbb{R}),
	\]
	then there exist two constants
	$c_0,c_1 \in \mathbb{R}$
	such that $u(t) = c_1 t + c_0$
	a.e. in $[a,b]$.
\end{lemma}
\begin{proof}
	Setting 
	\[
		A([a,b],\mathbb{R}) = \big\{u \in L^2([a,b],\mathbb{R}):
			u(t) = c_1 s + c_0,\,
			\text{for some }c_0,c_1\in\mathbb{R}
		\big\},
	\]
	by definition of $W_0([a,b],\mathbb{R})$
	we have that $A([a,b],\mathbb{R})$ and $W_0([a,b],\mathbb{R})$
	are orthogonal in $L^2([a,b],\mathbb{R})$.
	To conclude the proof, let us show that 
	$L^2([a,b],\mathbb{R}) = A([a,b],\mathbb{R}) \oplus W_0([a,b],\mathbb{R})$.
	This can be achieved by proving
	that for every $u \in L^2([a,b],\mathbb{R})$, 
	there exist exactly two real numbers $c_0, c_1$
	such that 
	\[
		u - c_1 t - c_0 \in W_0([a,b],\mathbb{R}).
	\]
	A straightforward computation shows that $c_0,c_1$
	must solve the following system:
	\[
		\begin{dcases}
			c_1 \frac{(b-a)^2}{2} + c_0 (b-a)
			= \int_{a}^b u(t)\, \mathrm{d}t,\\
			c_1 \frac{(b-a)^3}{3} + c_0 \frac{(b-a)^2}{2}
			= \int_{a}^b t\, u(t)\, \mathrm{d}t,
		\end{dcases}
	\]
	whose determinant is different from $0$ if $a\ne b$,
	and so $c_0$ and $c_1$ are unique.
\end{proof}

By directly applying Lemma~\ref{lem:affine-ortho},
we have the following regularity result in the 
one dimensional setting.
\begin{proposition}
	\label{prop:onedim-regularity}
	Let $\gamma \in \Gamma$ be a critical point for $f$,
	where $\Gamma$ is defined by~\eqref{eq:def-X}.
	Then, for every $i=1,\dots,N$
	the restriction $\gamma\big|_{[t_{i-1},t_{i}]}$
	is smooth and 
	\begin{equation}
		\label{eq:Euler-Lagrange}
		\frac{\mathrm{d}^4}{\mathrm{d}t^4}\gamma(t) = 0,
		\qquad \forall t \in (t_{i-1},t_{i}).
	\end{equation}
	Moreover, 
	both $\gamma\big|_{[0,t_j]}$
	and $\gamma\big|_{[t_j,1]}$
	are of class $C^2$;
	if $t_j \ne 0$ then $\ddot{\gamma}(0) = 0$
	and 
	if $t_j \ne 1$ then $\ddot{\gamma}(1) = 0$.
\end{proposition}
\begin{proof}
	As an initial step,
	let us select an index $i \in \{1,\dots,N\}$ 
	and demonstrate that the restriction
	of $\gamma$ to $[t_{i-1},t_{i}]$ is smooth. 
	We consider the subset of $V_0$ consisting
	of variations $\xi$ with compact support within
	the interval $[t_{i-1},t_{i}]$. 
	Given that $\gamma$ is a critical point for the functional $f$, 
	it follows that
	\[
		\mathrm{d}f(\gamma)[\xi]
		= \int_{t_{i-1}}^{t_i} \ddot{\gamma}(t)\, \ddot{\xi}(t)\, \mathrm{d}t
		= 0,
	\]
	for every $\xi$ within this particular subset of $V_0$.
	As $\xi$ is an element of $V_0$, 
	according to \eqref{eq:V_0-otherChar}
	$\ddot{\xi}$ belongs to $W_0([t_{i-1},t_{i}],\mathbb{R})$. 
	Conversely, for any $\eta \in W_0([t_{i-1},t_{i}],\mathbb{R})$, 
	there exists a $\xi$ in this subset of $V_0$ such that
	$\ddot{\xi}(t) = \eta(t)$ for all $t \in [t_{i-1},t_{i}]$.
	This equivalence is achieved 
	by integrating $\eta$ twice
	and adopting zero as the integration constants.
	By Lemma~\ref{lem:affine-ortho}
	we obtain two constants
	$c_0^i,c_1^i\in\mathbb{R}$
	such that 
	\[
		\ddot{\gamma}(t) = c_1^i t + c_0^i,
		\qquad \text{a.e. in } (t_{i-1},t_{i}).
	\]
	Therefore, $\gamma\big|_{[t_{i-1},t_{i}]}$
	is smooth and~\eqref{eq:Euler-Lagrange}
	holds for every $i = 0,\dots,N-1$.

	To demonstrate the final part of the proposition, 
	consider any $\xi \in V$. Given that $\gamma \in \Gamma$ is a critical point 
	and exhibits smoothness over each interval $[t_{i-1}, t_{i}]$, 
	integrating by parts within each interval yields 
	to the following equation:
	\begin{equation*}
		\mathrm{d}f(\gamma)[\xi]
		= \sum_{i = 1}^{N}
		\int_{t_{i-1}}^{t_i} \ddot{\gamma}(t)\ddot{\xi}(t)\,
		\mathrm{d}t
		= \sum_{i = 1}^{N}
		\left(
			\ddot{\gamma}(t_{i}^-)\dot\xi(t_{i})
			-
			\ddot{\gamma}(t_{i-1}^+)\dot\xi(t_{i-1})
		\right)
		= 0, 
		\qquad \forall \xi \in V,
	\end{equation*}
	where $\ddot{\gamma}(t^-)$ and $\ddot{\gamma}(t^+)$
	denote the second-order left and right derivatives
	of $\gamma$ at time $t$, respectively.
	The integral terms vanish because
	$\mathrm{d}^3\gamma/\mathrm{d}t^3$ remains constant 
	and $\dot{\xi}$ has a zero mean value
	over each interval.

	Given that for any $i \ne j$, there exists $\xi \in V$ such that 
	$\dot{\xi}(t_i) \ne 0$ and $\dot{\xi}(t_k) = 0$ for all $k \ne i$, 
	it follows that
	\[
		\ddot{\gamma}(t_{i}^-) = \ddot{\gamma}(t_{i}^+),
		\quad \text{if } i \ne j,
	\]
	which indicates that both sections of $\gamma$, specifically 
	$\gamma\big|_{[0,t_j]}$ and $\gamma\big|_{[t_j,1]}$, 
	are of class $C^2$.
	When $t_j$ does not coincide with $0$ and $1$,
	we deduce that $\ddot{\gamma}(0) = 0$ 
	and $\ddot{\gamma}(1) = 0$,
	respectively.
\end{proof}

\section{The regularity result in the Riemannian setting}
\label{sec:Riemannian-regularity}
In this section,
we prove Theorem~\ref{theorem:regularity},
thus generalizing the results achieved by
Proposition~\ref{prop:onedim-regularity}
in the one-dimensional setting
to the more general context of
Riemannian manifolds.

From now on, let $(M^n,g)$ is a smooth
$n$--dimensional complete Riemannian manifold, endowed 
with a smooth metric tensor $g$.
For any absolute continuous curve $\gamma\colon [a,b] \to M$
and any vector field $\eta$ along $\gamma$ of class $L^1$,
we denote by 
\[
	\covI_{a}^b \eta(t)\mathrm{d}t
\]
the \emph{covariant integral of $\eta$ along $\gamma$},
that is the unique vector field $\mu$ along $\gamma$ such that 
$\mu(a) = 0$ and $D_t\mu(t) = \eta(t)$ for every $t\in [a,b]$.
Hence, for any vector field $\xi$
of class $H^1$ along $\gamma$, we have
\begin{equation}
	\label{eq:covariantIntegral}
	\xi(t) = \xi_t(a)
	+ \covI_{a}^{t}\covD_t \xi(\tau)\mathrm{d}\tau,
\end{equation}
where we denote by $\xi_t(a) \in T_{\gamma(t)}M$
the parallel transport of the vector
$\xi(a) \in T_{\gamma(a)}M$
along the curve $\gamma$ at time $t \in [a,b]$.

Let $\Gamma\subset H^2([0,1],M)$ and $f\colon \Gamma \to \mathbb{R}$
be defined as in Theorem~\ref{theorem:regularity}.
Due to the constraints of the curves we are considering,
the set of admissible variations depends on each 
$\gamma \in \Gamma$ and is the following:
\begin{equation}
	\begin{aligned}
		\label{eq:def-V-Riemann}
		V_\gamma \coloneqq \Big\{ \xi \in H^2([0,1],TM) :\, 
		  & \xi(t) \in T_{\gamma(t)}M, \, \forall t \in [0,1], \\
		  & \xi(t_i) = 0,\, \forall i = 0,\dots, N,\, 
		  & \covD_t\xi(t_j) = 0 
	  \Big\}.
	\end{aligned}
\end{equation}
In other words, $V_\gamma$ is the set of
vector fields along $\gamma$
of Sobolev class $H^2$ vanishing at each time $t_i$,
for $i = 0,\dots,N$,
and with first covariant derivative equal to $0$ at time $t_j$.

We are concerning about the regularity and existence 
of the critical points of $f\colon \Gamma \to \mathbb{R}$, 
hence of those curves such that 
\[
	\mathrm{d}f(\gamma)[\xi] = 0,
	\qquad \forall \xi  \in V_\gamma.
\]
By standard arguments (see, e.g.,~\cite{GGP2002}),
it can be proved that the differential 
$\mathrm{d}f(\gamma)\colon V_\gamma \to \mathbb{R}$
reads as follows:
\begin{equation}
	\label{eq:diff-f-Riemannian}
	\mathrm{d}f(\gamma)[\xi]
	= \int_{0}^{1}
	\Big(
		g\big(\covD_t\dot{\gamma},\covD^2_t\,\xi\big)
		- g\big(R\big(\dot{\gamma},\covD_t\dot{\gamma}\big)\dot{\gamma},\xi\big)
	\Big) \mathrm{d}t,
\end{equation}
where $R\colon TM \times TM \times TM \to TM$
denotes the Riemannian curvature tensor
of the manifold.
Even in this case, due to the constraints
on the admissible variations,
we can't employ the DuBois Raymond's lemma to 
prove the regularity of the critical points.
Similarly to the one dimensional case, we proceed as follows.
For every $\gamma \in \Gamma$, let us consider
\[
	V_{\gamma,0} = \left\{\xi\in V_\gamma: \covD_t\,\xi(t_i) = 0,\,
	\forall i = 0,\dots, N\right\},
\]
and let us define the following vector field along $\gamma$,
\begin{equation}
	\label{eq:def-etax}
	\eta_\gamma(t) 
	\coloneqq
	\covI_0^t
	\left( \covI_0^\tau
		R\big(\dot{\gamma},\covD_s\dot{\gamma}\big)\dot{\gamma}\,
		\mathrm{d}s \right)\mathrm{d}\tau,
\end{equation}
so that $\eta_\gamma(t) \in  T_{\gamma(t)}M$ for every
$t \in [0,1]$ and the second order covariant
derivative of $\eta_\gamma$
is $R\big(\dot{\gamma},\covD_t\dot{\gamma}\big)\dot{\gamma}$.
With this notation, by integrating by parts (in the admissible direction)
we obtain
\begin{equation}
	\label{eq:diff-f-Riem-xi0}
	\mathrm{d}f(\gamma)[\xi]
	= \int_{0}^{1}
	g\left(\covD_t\dot{\gamma} - \eta_\gamma(t),
	\mathrm{D}^2_t\,\xi\right)
	\mathrm{d}t,
	\qquad \forall \xi \in V_{\gamma,0}.
\end{equation}
We are going to prove that, if $\mathrm{d}f(\gamma)[\xi] = 0$
for every $\xi \in V_{\gamma,0} \subset V_\gamma$,
in particular if $\gamma$ is a critical point,
then $\gamma$ is smooth on every time interval $[t_i,t_{i+1}]$.
As a first step, we give the following characterization result. 
\begin{lemma}
	\label{lem:characterize-Wx0}
	Let $\gamma \in H^2([a,b],M)$
	and let $\xi \in V_{\gamma,0}([a,b],\mathbb{R})$.
	Then, for every $i = 1,\dots,N$ we have
	\begin{equation}
		\label{eq:characterize-Wx0-firstEq}
		\covI_{t_{i-1}}^{t_i} \covD^2_t\xi\,\mathrm{d}t = 0,
	\end{equation}
	and 
	\begin{equation}
		\label{eq:characterize-Wx0-secondEq}
		\covI_{t_{i-1}}^{t_i} t\, \covD^2_t\xi\, \mathrm{d}t = 0.
	\end{equation}
\end{lemma}
\begin{proof}
	Since $\covD_t\xi(t_i) = 0$ for any $i = 0,\dots,N$,
	~\eqref{eq:characterize-Wx0-firstEq} 
	directly follows from~\eqref{eq:covariantIntegral}.
	Now, let us fix $i = 1,\dots, N$.
	Since $\xi(t_{i-1}) = \xi(t_i)  = 0$,
	and $\covD_t\xi(t_{i-1}) = 0$,
	we have
	\[
		0 = 
		\covI_{t_{i-1}}^{t_i}
		\covD_t \xi(t)\mathrm{d}t
		=
		\covI_{t_{i-1}}^{t_i}
		\left(
			\covI_{t_{i-1}}^{t}
			\covD^2_\tau\xi(\tau)
			\mathrm{d}\tau
		\right)
		\mathrm{d}t.
	\]
	Then, by using an integration by parts, we
	obtain
	\[
		\covI_{t_{i-1}}^{t_i}
		t\,
		\covD_\tau^2\xi(t)
		\mathrm{d}t
		= 
		-
		\covI_{t_{i-1}}^{t_i}
		\left(
			\covI_{t_i}^{t}
			\covD^2_\tau\xi(\tau)
			\mathrm{d}\tau
		\right)
		\mathrm{d}t = 0.
	\]
\end{proof}

As a consequence of the previous lemma,
let us introduce the following notation.
Let $[a,b]\subset \mathbb{R}$ and let $\gamma\colon [a,b] \to M$
be a curve of class $H^2$.
Defining $W_{\gamma,0}([a,b],M)$ as follows,
\[
	W_{\gamma,0}([a,b],M)
	\coloneqq\Big\{
		\eta\in L^2([a,b],TM):\,
		\eta(t) \in T_{\gamma(t)}M,\,
		\covI_a^b \eta(t)\,\mathrm{d}t = 0,
		\covI_a^b t,\ \eta(t)\,\mathrm{d}t = 0
	\Big\},
\]
we are ready to state the generalization of 
the DuBois-Reymond lemma in the Riemannian setting.

\begin{lemma}
	\label{lem:second-order-DuBois-Riemann}
	If $\mu\colon [a,b] \to TM$ is a $L^1$ vector
	field along a continuous curve $\gamma\colon [a,b] \to M$ such that 
	\[
		\int_{a}^{b}g\left(\mu,\eta\right)\mathrm{d}t
		= 0,
		\qquad \forall \eta \in W_{\gamma,0}([a,b],M),
	\]
	then there exist two parallel vector fields
	$\nu$ and $\zeta$
	along $\gamma$ such that 
	\begin{equation}
		\label{eq:mu-affine-comb}
		\mu(t) = \nu(t) + t \zeta(t),
		\qquad \text{a.e. in }(a,b).
	\end{equation}
\end{lemma}
\begin{proof}
	By definition of $W_{\gamma,0}$, any $\mu$ given
	by~\eqref{eq:mu-affine-comb},
	with $\nu$ and $\zeta$ two arbitrary parallel
	vector fields along $\gamma$,
	is orthogonal to $W_{\gamma,0}$.
	Indeed, for any $\eta \in W_{\gamma,0}$ and
	a vector field $\nu$ along the curve $\gamma$ 
	with zero covariant derivative we have
	\[
		\int_{a}^{b} g(\nu,\eta)\mathrm{d}t
		= g\Big(\nu(b),\covI_a^b \eta\,\mathrm{d}t\Big)
		- \int_{a}^b g\Big(\covD_t\nu,
		\covI_a^t\eta\,\mathrm{d}s\Big)\mathrm{d}t
		= 0,
	\]
	since the both terms vanish.
	Similarly, 
	for any vector field $\zeta$ along $\gamma$ 
	with zero covariant derivative we obtain
	\begin{equation*}
		\int_{a}^{b} g( t \zeta,\eta)\mathrm{d}t
		=
		\int_{a}^{b} g(  \zeta,t\eta)\mathrm{d}t
		= g\Big(\zeta(b),\covI_a^b t \eta\,\mathrm{d}t\Big)
		- \int_a^bg\Big(\covD_t\zeta,\covI_a^t s
		\eta\,\mathrm{d}s\Big)\mathrm{d}t 
		= 0.
	\end{equation*}
	Therefore, it remains to prove that any
	arbitrary vector field $\rho$ along $\gamma$ can be obtained
	as a direct sum of an element in $W_{\gamma,0}$
	and one given by~\eqref{eq:mu-affine-comb}.
	In other words, we need to show the existence
	and uniqueness of two parallel
	vector fields $\nu$ and $\zeta$ along $\gamma$ such that
	$\rho - \nu - t\zeta$ belongs to $W_{\gamma,0}$,
	namely
	\[
		\covI_{a}^b (\nu + t\zeta) \mathrm{d}t 
		= \covI_{a}^b \rho\,\mathrm{d}t
		\qquad\text{and}\qquad
		\covI_a^b(t\nu + t^2\zeta) \mathrm{d}t 
		= \covI_{a}^b t \rho\,\mathrm{d}t.
	\]
	Denoting by $\nu_b,\zeta_b \in T_{\gamma(b)}M$
	the two vectors that uniquely determined the 
	parallel vector fields $\nu$ and $\zeta$, respectively,
	we have that $\rho - \nu - t\zeta$
	belongs to $W_{\gamma,0}$ if and only if there exist
	$\nu_b,\zeta_b \in T_{\gamma(b)}M$ such that
	\[
		\begin{dcases}
			(b-a)\nu_b + \frac{(b-a)^2}{2}\zeta_b
			= \covI_{a}^b \rho\,\mathrm{d}t,\\
			\frac{(b-a)^2}{2}\nu_b + \frac{(b-a)^3}{3}\zeta_b
			= \covI_{a}^b t \rho\,\mathrm{d}t.
		\end{dcases}
	\]
	Since the last system uniquely provides the two vectors $\nu_b,\zeta_b$
	(unless $b \ne a$), the thesis is obtained.
\end{proof}

\begin{proof}[Proof of Theorem~\ref{theorem:regularity}]
	The main steps of the proof are the same as those of
	Proposition~\ref{prop:onedim-regularity};
	however, in this setting, we can appreciate the iterative procedure
	of the bootstrap method.
	
	Let us fix a time interval
	$[t_{i-1},t_i]$, with $i = 1,\dots,N$.
	Since $\gamma \in \Gamma$ is a critical point,
	for any $\xi \in V_{\gamma,0}$
	with compact support in $(t_{i-1},t_i)$,
	by~\eqref{eq:diff-f-Riem-xi0}
	we have
	\[
		\mathrm{d}f(\gamma)[\xi] 
		= \int_{t_{i-1}}^{t_i}
		g\left(\covD_t\dot{\gamma} - \eta_\gamma(t),
		\mathrm{D}^2_t\,\xi\right)
		\mathrm{d}t = 0,
	\]
	where we recall that $\eta_{\gamma}$
	is the vector field along $\gamma$
	given by~\eqref{eq:def-etax}.
	Using Lemma~\ref{lem:characterize-Wx0},
	we have that $\covD^2_t\xi \in W_{\gamma,0}([t_{i-1},t_i],M)$.
	Therefore,
	by the generality of $\xi$
	and using Lemma~\ref{lem:second-order-DuBois-Riemann},
	we have the existence of two parallel
	vector fields $\nu^i,\zeta^i$ along $\gamma|_{[t_{i-1},t_i]}$
	such that
	\begin{equation}
		\label{eq:bootstrap-Riemannian}
		\covD_t\dot{\gamma}(t) - \eta_{\gamma}(t) = \nu^i(t) + t\zeta^i(t),
		\qquad \text{a.e. in } (t_{i-1},t_i).
	\end{equation}
	Now, the iterative procedure of the bootstrap
	method can start.
	Since $\eta_{\gamma}$ is of class $H^2$
	and $\nu^i$ and $\zeta^i$ are parallel vector fields,
	from the previous equation we have
	that $\covD_t\dot\gamma$ has $H^2$--regularity
	on $[t_{i-1},t_i]$,
	hence $\gamma$ is of class $H^4$
	on this interval.
	Since by definition of $\eta_\gamma$
	we have
	\begin{equation}
		\label{eq:bootstrap-Riemannian-2}
		\covD^2_t\eta_\gamma(t)
		= R\big(\dot{\gamma}(t),\covD_t\dot{\gamma}(t)\big)
		\dot{\gamma}(t),
		\qquad \forall t \in [t_{i-1},t_i],
	\end{equation}
	and $\gamma|_{[t_{i-1},t_i]}$ is of class $H^4$,
	we have that $\eta_{\gamma}$
	is actually of class $H^4$ and,
	by using again~\eqref{eq:bootstrap-Riemannian},
	$\gamma|_{[t_{i-1},t_i]}$ is of class $H^6$.
	By applying alternatively~\eqref{eq:bootstrap-Riemannian}
	and~\eqref{eq:bootstrap-Riemannian-2},
	we have that $\gamma|_{[t_{i-1},t_i]}$
	belongs to $H^{2k}([t_{i-1},t_i],M)
	\subset C^{2k-1}([t_{i-1},t_i],M)$
	for every $k \in \mathbb{N}$,
	hence it is smooth on $[t_{i-1},t_i]$.
	Since the above procedure can be applied
	for every $i = 1,\dots,N$,
	the desired regularity result is achieved.
	As a consequence, by taking the second order
	covariant derivative on both sides
	of~\eqref{eq:bootstrap-Riemannian},
	we obtain the spline equation~\eqref{eq:Riemannian-spline}.
	
	To prove the last part of the theorem,
	it suffices to compute the first variation
	of $\gamma$ for every vector field $\xi \in V_{\gamma}$.
	By the above regularity result,
	we can integrate by parts the first variation formula
	and, recalling that $\xi(t_i) = 0$
	for every $i = 0,\ldots,N$,
	and that~\eqref{eq:Riemannian-spline}
	holds on every interval $[t_{i-1},t_i]$,
	we obtain
	\begin{multline}
		\mathrm{d}f(\gamma)[\xi]
		=
		\sum_{i = 1}^N
		\int_{t_{i-1}}^{t_i}
		\Big(
		g\big(\covD_t\dot{\gamma},\covD^2_t\,\xi\big)
		- g\big(R\big(\dot{\gamma},\covD_t\dot{\gamma}\big)\dot{\gamma},\xi\big)
		\Big) \mathrm{d}t
		\\
		=
		\sum_{i = 1}^N
		g\big(\covD_t\dot{\gamma},\covD_t\xi\big)\big|_{t_{i-1}}^{t_i}
		-
		\int_{t_{i-1}}^{t_i}
		\Big(
		g\big(\covD^2_t\dot{\gamma},\covD_t\,\xi\big)
		+ g\big(R\big(\dot{\gamma},\covD_t\dot{\gamma}\big)\dot{\gamma},\xi\big)
		\Big) \mathrm{d}t
		\\
		=
		\sum_{i = 1}^N
		g\big(\covD_t\dot{\gamma},\covD_t\xi\big)\big|_{t_{i-1}}^{t_i}
		+
		\int_{t_{i-1}}^{t_i}
		g\big(\covD^3_t\dot{\gamma} -
		R\big(\dot{\gamma},\covD_t\dot{\gamma}\big)\dot{\gamma},\xi\big)
		\mathrm{d}t
		\\
		=
		\sum_{i = 1}^{N}
		\left(
		g\big(\covD_t\dot\gamma(t_i^-),\covD_t\xi(t_i^-)\big)
		-
		g\big(\covD_t\dot\gamma(t_{i-1}^+),\covD_t\xi(t_{i-1}^+)\big)
		\right)
		= 0.
	\end{multline}
	Since for every index $i \ne j$, $i \ne 0,N$, we can choose 
	$\xi \in V_{\gamma}$ such that $\covD_t\xi$ 
	does not vanish only at that index,
	for such a variation we obtain
	\[
		g\big(\covD_t\dot{\gamma}(t_i^-) - \covD_t\dot{\gamma}(t_i^+),
		\covD_t\xi(t_i)\big) = 0,
	\]
	and by the arbitrariness of $\xi$ we obtain 
	the joining condition $\covD_t\dot\gamma(t_i^-) = \covD_t\dot\gamma(t_i^+)$,
	for every $i \ne j$.
	This implies that both $\gamma|_{[0,t_j]}$
	and $\gamma|_{[t_j,1]}$ are of class $C^2$.
	Finally, if $j \ne 0$ or $j \ne 1$, then an analogous argument shows
	that $\covD_t\dot{\gamma}(0) = 0$
	or $\covD_t\dot{\gamma}(1) = 0$, respectively,
	and we are done.
\end{proof}

\section{Regularity of $k$--splines}
\label{sec:k-regularity}
This section generalizes the previous regularity 
results for the $k$--splines, for every integer $k \ge 2$.
As we did for the splines,
at first we show it in the one dimensional setting,
where the notation is not overwhelming,
and then we give it in the Riemannian setting,
thus proving Theorem~\ref{theorem:regularity-k-Riemann}.

For the more general case of the $k$--splines,
Lemma~\ref{lem:second-order-DuBois-Riemann} is not sufficient
to start the bootstrap method if $k \ge 3$, hence a generalization
of that result is required to obtained the desired regularity,
which in the one dimensional is the following.
\begin{lemma}
	\label{lem:DuBois-k}
	Let $k$ be a non-negative integer.
	Let $u \in L^1([a,b],\mathbb{R})$
	such that
	\[
		\int_a^b u(t)\eta(t)\mathrm{d}t = 0,
	\]
	for any $\eta \in C^{\infty}_c((a,b),\mathbb{R})$
	that satisfies
	\[
		\int_a^b t^\ell \eta(t)\mathrm{d}t = 0,
		\qquad \forall \ell = 0,\dots,k.
	\]
	Then, there exist $k+1$ constants
	$c_0,\dots,c_k \in \mathbb{R}$
	such that
	\[
		u(t) = c_k t^k + \dots + c_1 t + c_0
		= \sum_{\ell = 0}^k c_\ell t^\ell,
		\qquad \text{a.e. in }[a,b].
	\]
\end{lemma}

\begin{proof}
	The proof relies on an induction argument.
	When $k = 0$, the statement
	is nothing but the DuBois-Reymond Lemma.

	Assume that the thesis holds for $k$
	and let us show it for $k + 1$.
	Let $\xi \in C^{\infty}_c((a,b),\mathbb{R})$
	be such that
	\[
		\int_a^b t^\ell \xi(t)\mathrm{d}t = 0,
		\qquad \forall \ell = 0,\dots,k,
	\]
	and we fix a function $\phi \in C^{\infty}_c((a,b),\mathbb{R})$
	such that $\int_a^b \phi(t)\mathrm{d}t = 1$.
	Moreover, let us define
	\[
		c =
		\frac{(-1)^{k}}{(k+1)!}
		\int_a^b t^{k+1}\xi(t)\mathrm{d}t \in \mathbb{R}.
	\]
	
	Then, we define $\eta \in C^{\infty}_c((a,b),\mathbb{R})$
	as follows:
	\[
		\eta(t) = \xi(t) + c\, \phi^{(k+1)}(t),
	\]
	where
	\[
		\phi^{(k+1)}(t) = \frac{\mathrm{d}^{k+1}}{\mathrm{d}t^{k+1}}\phi(t).
	\]
	By using the integration by parts formula
	$\ell$ times, for every $\ell = 0,\dots,k$
	we have
	\begin{multline*}
		\int_a^b t^\ell \eta(t)\mathrm{d}t
		= 
		\int_a^b t^\ell \xi(t)\mathrm{d}t
		+ c 
		\int_a^b t^\ell \phi^{(k+1)}(t)\mathrm{d}t\\
		= (-1)^{\ell}
		c\,\ell! \int_a^b \phi^{(k+1-\ell)}(t)
		\mathrm{d}t
		= (-1)^{\ell}
		c\,\ell! \big(\phi^{(k-\ell)}(b) - 
			\phi^{(k-\ell)}(a)  \big)
		=  0.
	\end{multline*}
	By the same procedure and by definition of $c$,
	we obtain
	\begin{multline*}
		\int_a^b t^{k+1}\eta(t)\mathrm{d}t
		= 
		\int_a^b t^{k+1}\xi(t)\mathrm{d}t
		+ c 
		\int_a^b t^{k+1} \phi^{(k+1)}(t)\mathrm{d}t\\
		=
		\int_a^b t^{k+1}\xi(t)\mathrm{d}t
		+
		(-1)^{k+1}
		c\,(k+1)! \int_a^b \phi(t)
		\mathrm{d}t
		= \int_a^b t^{k+1}\xi(t)\mathrm{d}t
		+
		(-1)^{k+1} c\,(k+1)!
		= 0.
	\end{multline*}
	Hence, we have that $\int_a^b t^\ell \eta(t)\mathrm{d}t = 0$
	for any $\ell = 0,\dots, k+1$,
	and by hypothesis we have  
	\[
		\int_a^b u(t)\eta(t)\mathrm{d}t = 0.
	\]
	Setting, 
	\[
		c_{k+1} = - \frac{(-1)^{k}}{(k+1)!}\int_a^b u(t)\phi^{(k+1)}(t)\mathrm{d}t
	\]
	we can expand the previous equality as follows:
	\begin{multline*}
		\int_a^b u(t)\eta(t)\mathrm{d}t 
		= \int_a^b u(t) \xi(t)\mathrm{d}t
		+ c \int_a^b u(t)\phi^{(k+1)}(t)\mathrm{d}t\\
		= \int_a^b u(t) \xi(t)\mathrm{d}t +
		\left(
		\frac{(-1)^{k}}{(k+1)!}
		\int_a^b t^{k+1}\xi(t)\mathrm{d}t
	\right)\int_a^b u(t)\phi^{(k+1)}(t)\mathrm{d}t\\
	=
	\int_a^b \big(u(t) - c_{k+1}t^{k+1}\big)\xi(t)\mathrm{d}t
	= 0.
	\end{multline*}
	By the arbitrariness of $\xi$ and by the induction hypothesis, 
	we then obtain the existence of $k+1$ constants
	$c_0,\dots,c_k \in \mathbb{R}$ such that
	\[
		u(t) - c_{k+1}t^{k+1} = c_k t^k + \dots + c_0,
		\qquad \text{a.e. in }(a,b).
	\]
	Therefore, the thesis holds also for $k+1$
	and this ends the proof.
\end{proof}

Thanks to the previous generalization, we have the following
regularity result in the one dimensional setting.
\begin{proposition}
	\label{prop:regularity-k}
	Let $k$ be a positive integer
	and let $v_1,\dots,v_{k-1}\in \mathbb{R}$.
	Let us also fix $p_0,\dots,p_N \in \mathbb{R}$
	for some positive integer $N$
	and let us fix $j \in \{0,\dots,N\}$.
	Setting 
	\[
	\Gamma_k =
	\left\{\gamma \in H^{k}([0,1],\mathbb{R}):
		\gamma(t_i) = p_i, \, \forall i = 0, \ldots, N,
		\gamma^{(\ell)}(t_j) = v_\ell,
		\, \forall \ell = 1,\dots,k-1
	\right\}
	\]
	and $f_k\colon \Gamma_k \to \mathbb{R}$ as follows:
	\[
		f_k(\gamma) = \int_0^1
		|\gamma^{(k)}(t)|^2\mathrm{d}t.
	\]
	If $\gamma$ is a critical point of $f_k$,
	then, for every $i=1,\dots,N$
	the restriction $\gamma\big|_{[t_{i-1},t_{i}]}$
	is smooth and 
	\begin{equation}
		\label{eq:Euler-Lagrange-k}
		\frac{\mathrm{d}^{2k}}{\mathrm{d}t^{2k}}\gamma(t) = 0,
		\qquad \forall t \in [t_{i-1},t_{i}].
	\end{equation}
	Moreover, 
	both $\gamma\big|_{[0,t_j]}$
	and $\gamma\big|_{[t_j,1]}$
	are of class $C^{2k-2}$;
	if $t_j \ne 0$ then $\gamma^{(k+\ell-1)}(0) = 0$
	for every $\ell = 1,\dots,k-1$,
	and similarly if $t_j \ne 1$.
\end{proposition}

\begin{proof}
	By standard computations, 
	it can be shown that
	for every $\gamma \in \Gamma_k$,
	the set of admissible variations
	is given by
	\[
		V_k = \left\{
			\xi \in H^k([0,1],\mathbb{R}):
			\xi(t_i) = 0,\, \forall i = 0,\dots,N,
			\,
			\xi^{(\ell)}(t_j) = 0,
			\, \forall \ell = 1,\dots,k-1
		\right\}.
	\]
	Let us fix $i = 1,\dots,N$
	and a variation $\xi \in V$
	such that $\supp \xi \subset (t_{i-1},t_i)$.
	Therefore, $\xi^{(\ell)}(t_{i-1}) = \xi^{(\ell)}(t_i) = 0$
	for every $\ell = 0,\dots,k-1$,
	and this implies that
	\begin{multline*}
		\int_{t_{i-1}}^{t_i}
		t^\ell \xi^{(k)}(t) \mathrm{d}t
		= (-1)^\ell\,\ell!
		\int_{t_{i-1}}^{t_i}
		\xi^{(k-\ell)}(t) \mathrm{d}t\\
		= (-1)^\ell\,\ell!
		\big(
			\xi^{(k-\ell-1)}(t_i)
			- \xi^{(k-\ell-1)}(t_{i-1})
		\big)
		= 0,
		\qquad \forall \ell = 1,\dots,k-1.
	\end{multline*}
	By a standard computation, we have
	\[
		\mathrm{d}f_k(\gamma)[\xi]
		= \int_{t_{i-1}}^{t_i}
		\gamma^{(k)}(t) \xi^{(k)}(t) \mathrm{d}t
		= 0,
	\]
	and therefore, by Lemma~\ref{lem:DuBois-k},
	we have that 
	\[
		\gamma^{(k)}(t) = c_{k-1}^i t^{k-1} + \dots + c_1^i t + c_0^i,
		\quad \text{a.e. in } (t_{i-1},t_i),
	\]
	for some real constants $c_{k-1}^i, \dots, c_0^i$.
	Hence, $\gamma|_{[t_{i-1},t_i]}$
	is smooth and~\eqref{eq:Euler-Lagrange-k} holds,
	for every $i = 1,\dots,N$.

	Now, let $\xi$ be a general variation.
	A careful application of the integration by parts formula
	leads to the following equality:
	\begin{equation}
		\label{eq:firstVar-k-smooth}
		\mathrm{d}f_k(\gamma)[\xi]
		= \sum_{i = 1}^{N} \left( \sum_{\ell = 1}^{k-1}
			(-1)^{\ell}
			\Big(\gamma^{(k+\ell-1)}(t_i^-) \xi^{(k-\ell)}(t_i)-
				\gamma^{(k+\ell-1)}(t_{i-1}^+) \xi^{(k-\ell)}(t_{i-1})
		\Big) \right) = 0.
	\end{equation}
	For every $i \ne j$, $i \ne 0,N$,
	and for every $\ell = 1,\dots,k-1$,
	there exists $\xi \in V_k$ such that
	$\xi^{(k-\ell)}(t_i) \ne 0$
	only for these $i$ and $\ell$.
	In this case we have
	\[
		\mathrm{d}f_k(\gamma)[\xi]
		= (-1)^{\ell}
		\Big(\gamma^{(k+\ell-1)}(t_i^-) - 
			\gamma^{(k+\ell-1)}(t_{i}^+)
		\Big) \xi^{(k-\ell)}(t_{i}) = 0,
	\]
	meaning that 
	\[
		\gamma^{(k+\ell-1)}(t_i^-) = \gamma^{(k+\ell-1)}(t_{i}^+),
		\qquad \forall \ell = 1,\dots,k-1,
		\, \forall i \ne j, i \ne 0,N,
	\]
	from which we infer that both
	$\gamma|_{[0,t_j]}$ and $\gamma|_{[t_j,1]}$
	are of class $C^{2k-2}$.
	If $t_j \ne 0$, by choosing a variation $\xi \in V_k$
	such that $\supp \xi \subset [0,t_1)$
	we obtain 
	\[
		\mathrm{d}f_k(\gamma)[\xi]
		= \sum_{\ell = 1}^{k-1}
		\gamma^{(k+\ell-1)}(0)
		\xi^{(k-\ell)}(0)
		= 0,
	\]
	from which we infer
	that $\gamma^{(k+\ell-1)}(0) = 0$
	for every $\ell = 1,\dots,k-1$,
	and an analogous result can be obtained 
	if $t_j \ne 1$.
\end{proof}

\subsection{Regularity of $k$--splines in the Riemannian setting}

This section provides a proof of Theorem~\ref{theorem:regularity-k-Riemann}.
As for the above regularity
result in the one dimensional setting,
we need a generalization of the DuBois-Reymond Lemma,
which in this case is the following.
\begin{lemma}
	\label{lem:DuBois-k-Riemannian}
	Let $k$ be a non-negative integer.
	If $\mu\colon [a,b] \to TM$ is a vector
	field of class $L^1$ along an absolute continuous curve
	$\gamma\colon [a,b] \to M$ such that 
	\[
		\int_a^b g\big(\mu(t),\eta(t)\big)\mathrm{d}t = 0,
	\]
	for any $\eta \in C^{\infty}_c((a,b),\mathbb{R})$
	that satisfies
	\[
		\covI_a^b t^\ell \eta(t)\mathrm{d}t = 0,
		\qquad \forall \ell = 0,\dots,k,
	\]
	then there exist $k+1$ constants
	parallel vector fields along $\gamma$,
	denoted by $\zeta_{0},\dots,\zeta_k$, such that
	\begin{equation}
		\label{eq:thesis-DBR-k-riemannian}
		\mu(t) = t^k\zeta_k(t) + \dots + t \zeta_1(t)+ \zeta_0(t)
		= \sum_{\ell = 0}^k t^\ell \zeta_\ell(t),
		\qquad \text{a.e. in }(a,b).
	\end{equation}
\end{lemma}
\begin{proof}
	The proof follows the same lines
	of the one of Lemma~\ref{lem:DuBois-k},
	namely it uses an induction argument and
	a direct construction of the first term,
	namely of the parallel vector field
	that appears with the higher order in
	~\eqref{eq:thesis-DBR-k-riemannian}.

	If $k = 0$, then the thesis is the
	DuBois-Reymond Lemma in the Riemannian setting.
	Thus, let us assume that the thesis holds
	for $k$ and let us prove it for $k+1$.

	Let $\xi\colon [a,b] \to TM$
	be a smooth vector field along $\gamma$
	with compact support in $(a,b)$
	such that 
	\[
		\covI_{a}^b t^\ell\xi(t)\mathrm{d}t = 0,
		\qquad\forall \ell = 0,\dots,k.
	\]
	Let us fix a function $\phi\in C^{\infty}_c((a,b),\mathbb{R})$
	such that $\int_{a}^b \phi(t)\mathrm{d}t = 1$.
	Let $c_{\xi}\colon [a,b] \to TM$	
	be the parallel vector field along 
	$\gamma$ such that
	\[
		c_{\xi}(b) = \frac{(-1)^k}{(k+1)!}
		\covI_{a}^b t^{k+1}\xi(t)\mathrm{d}t
		\in T_{\gamma(b)}M.
	\]
	With this notation, let us define
	the vector field $\eta\colon [a,b] \to TM$
	as follows:
	\[
		\eta(t)\coloneqq \xi(t) + \phi^{(k+1)}(t) c_\xi(t).
	\]
	For any $\ell = 0,\dots,k$, we have
	\[
		\covI_{a}^b t^\ell \eta(t)\mathrm{d}t
		= 
		\covI_{a}^b t^\ell \xi(t)\mathrm{d}t
		+ 
		\covI_{a}^b \phi^{(k+1)}(t) t^\ell c_\xi(t)\mathrm{d}t
		= 
		(-1)^{\ell}\,\ell!
		\covI_{a}^b \phi^{(k+1-\ell)}(t)c_\xi(t)
		= 0,
	\]
	where, since $c_\xi$ is a parallel vector field,
	we have
	$\covD_t\big(t^\ell c_\xi(t)\big)
	= \ell\, t^{\ell-1} c_\xi(t)$.
	Moreover, we have
	\begin{multline*}
		\covI_{a}^b	t^{k+1} \eta(t)\mathrm{d}t
		= 
		\covI_a^b t^{k+1}\xi(t)\mathrm{d}t
		+ (-1)^{k+1} (k+1)!
		\covI_a^b \phi(t)c_{\xi}(t)\mathrm{d}t
		\\
		= 
		\covI_a^b t^{k+1}\xi(t)\mathrm{d}t
		+
		(-1)^{k+1}(k+1)!
		c_{\xi}(b)
		= 
		\covI_a^b t^{k+1}\xi(t)\mathrm{d}t
		- \covI_a^b t^{k+1}\xi(t)\mathrm{d}t
		= 0.
	\end{multline*}
	Therefore, by hypothesis, we have
	\[
		\int_{a}^b g(\mu(t),\eta(t))\mathrm{d}t
		= 0,
	\]
	hence
	\begin{equation}
		\label{eq:reg-k-Rieman-proof1}
		\int_{a}^b g\big(\mu(t),\xi(t) + \phi^{(k+1)}(t)c_\xi(t)\big)
		= 0.
	\end{equation}
	Let $\zeta_{k+1}\colon [a,b] \to TM$
	be the parallel vector field such that
	\[
		\zeta_{k+1}(b) = -\frac{(-1)^k}{(k+1)!}
		\covI_{a}^b \phi^{(k+1)}(t)\mu(t)\mathrm{d}t,
	\]
	and let us notice that
	\begin{multline*}
		\int_{a}^b g\big(\mu(t),\phi^{k+1}(t)c_\xi(t)\big)
		\mathrm{d}t
		=
		\int_{a}^b g\big(\phi^{k+1}(t)\mu(t),c_\xi(t)\big)
		\mathrm{d}t
		 \\ =
		g\Big(\covI_{a}^b \phi^{(k+1)}(t)\mu(t)\mathrm{d}t,c_\xi(b)\Big)
		= 
		g\Big(\covI_{a}^b \phi^{(k+1)}(t)\mu(t)\mathrm{d}t,\frac{(-1)^k}{(k+1)!}
		\covI_{a}^b t^{k+1}\xi(t)\mathrm{d}t\Big)\\
		= -g\Big(\zeta_{k+1}(b),\covI_{a}^b t^{k+1}\xi(t)\Big)
		=- \int_a^b g\big(\zeta_{k+1}(t),t^{k+1}\xi(t)\big)
		\mathrm{d}t
		=- \int_a^b g\big(t^{k+1}\zeta_{k+1}(t),\xi(t)\big)
		\mathrm{d}t.
	\end{multline*}
	Therefore, from~\eqref{eq:reg-k-Rieman-proof1}
	we obtain the following identity:
	\[
		\int_{a}^b g\big(\mu(t) - t^{k+1}\zeta_{k+1}(t),\xi(t)\big)
		= 0.
	\]
	By the arbitrariness of $\xi$ and the inductive hypothesis,
	there exist $k+1$ constant vector fields along $\gamma$,
	say $\zeta_0,\dots,\zeta_{k}$ such that
	\[
		\mu(t) = t^{k+1}\zeta_{k+1}(t)
		+ \dots + t\zeta_1(t) + \zeta_0(t),
		\qquad \text{a.e. in }[a,b],
	\]
	and we are done.
\end{proof}

\begin{proof}[Proof of Theorem~\ref{theorem:regularity-k-Riemann}]
	In the setting provided by Theorem~\ref{theorem:regularity-k-Riemann},
	for every $\gamma \in \Gamma_k$
	the set of admissible variations is
	\begin{equation}
		\label{eq:def-V-Riemann-k}
		V_{k,\gamma} \coloneqq \left\{
			\begin{aligned}
				\xi \in H^k([0,1],TM) :
				\xi(t) \in T_{\gamma(t)}M, \, \forall t \in [0,1],
		& \quad \xi(t_i) = 0,\, \forall i = 0,\dots, N,\\
		& \quad \covD_t^{\ell}\xi(t_j) = 0,\, \forall \ell = 1,\dots,k-1
			\end{aligned}
		\right\},
	\end{equation}
	and the differential $\mathrm{d}f_k(\gamma)\colon V_{k,\gamma} \to \mathbb{R}$
	is given by the following formula
	(cf. \cite[Proposition 3.1]{GGP2004}):
	\begin{equation}
		\label{eq:firstVariation-Riemannian-k}
		\mathrm{d}f_k(\gamma)[\xi]
		= 
		\int_{0}^1
		g\Big(\covD_t^{k-1}\dot{\gamma},
			\covD_t^{k} \xi
			+
			\sum_{\ell = 0}^{k-2}
			\covD_t^\ell
			\big(R(\xi,\dot{\gamma})
				\covD_t^{k - \ell -2}
				\dot{\gamma}
			\big)
		\Big)\mathrm{d}t.
	\end{equation}
	Due to the lengthy computations and 
	heavy notation involving
	the covariant derivatives of the curvature tensor $R$,
	we will first show the regularity result 
	for the case $k = 3$.
	This approach allows us to clearly present the key steps of the proof
	without being overwhelmed by excessive notation.
	We will then extend the results to higher values of $k$.

	Let $k = 3$; as usual,
	for a given $i = 1,\dots,N$,
	let us consider a variation $\xi \in V_{k,\gamma}$
	with compact support in $(t_{i-1},t_i)$.
	Then, by applying the symmetric and skew-symmetric 
	properties of the curvature tensor $R$ and of its
	covariant derivative, we have
	\begin{equation*}
		\begin{aligned}
			\mathrm{d}f_k(\gamma)[\xi]
		& = \int_{t_{i-1}}^{t_{i}}
		g\big(\covD_t^{2}\dot{\gamma},\covD_t^3\xi\big)
		\mathrm{d}t
		+
		\int_{t_{i-1}}^{t_{i}}
		g\big(\covD_t^{2}\dot{\gamma},
		R(\xi,\dot{\gamma})\covD_t \dot{\gamma} \big)
		\mathrm{d}t
		+
		\int_{t_{i-1}}^{t_{i}}
		g\big(\covD_t^{2}\dot{\gamma},
		\covD_t\big(R(\xi,\dot{\gamma})\covD_t \dot{\gamma}\big) \big)
		\mathrm{d}t\\
		&= 
		 \int_{t_{i-1}}^{t_{i}}
		g\big(\covD_t^{2}\dot{\gamma},\covD_t^3\xi\big)
		\mathrm{d}t
		+
		\int_{t_{i-1}}^{t_{i}}
		g\big(\covD_t^{2}\dot{\gamma},
		R(\xi,\dot{\gamma})\covD_t \dot{\gamma} \big)
		\mathrm{d}t\\
		&
		+
		\int_{t_{i-1}}^{t_i}
		g\Big(\covD_t^{2}\dot{\gamma},
			(\covD_t R)(\xi,\dot{\gamma})\dot{\gamma}
			+ R(\covD_t\xi,\dot{\gamma})\dot{\gamma}
			+ R(\xi,\covD_t\dot{\gamma})\dot{\gamma}
			+ R(\xi,\dot{\gamma})\covD_t\dot{\gamma}
		\Big)
		\mathrm{d}t\\
		&= \int_{t_{i-1}}^{t_{i}}
		g\big(\covD_t^{2}\dot{\gamma},\covD_t^3\xi\big)
		\mathrm{d}t
		- \int_{t_{i-1}}^{t_{i}}
		g\big(R(\covD_t\dot{\gamma},\covD_t^2\dot{\gamma})\dot{\gamma},
		\xi\big)
		\mathrm{d}t\\
		&- \int_{t_{i-1}}^{t_{i}}
		g\big((\covD_tR)(\dot{\gamma},
			\covD_t^2\dot{\gamma})\dot{\gamma}
			+ R(\dot{\gamma},\covD_t^2\dot{\gamma})\covD_t\dot{\gamma}
			+ R(\covD_t\dot{\gamma},\covD_t^2\dot{\gamma})\dot{\gamma},
		\xi\big)
		\mathrm{d}t\\
		&- \int_{t_{i-1}}^{t_{i}}
		g\big(
			R(\dot{\gamma},\covD_t^2\dot{\gamma})\dot{\gamma},
			\covD_t\xi
		\big)
		\mathrm{d}t.
		\end{aligned}	
	\end{equation*}
	For the sake of simplicity, we can rewrite
	the previous equation as follows:
	\[
		\mathrm{d}f_k(\gamma)[\xi]
		= 
		\int_{t_{i-1}}^{t_i}
		g\big(\covD_t^2\dot{\gamma},\covD_t^3\xi\big)\mathrm{d}t
		-
		\int_{t_{i-1}}^{t_i}
		g\big(h_3(\dot{\gamma},\covD_t\dot{\gamma},\covD_t^2\dot{\gamma}),
		\xi\big)\mathrm{d}t
		-
		\int_{t_{i-1}}^{t_i}
		g\big(h_2(\dot{\gamma},\covD_t^2\dot{\gamma}),
		\covD_t\xi\big)\mathrm{d}t,
	\]
	where both $h_3$ and $h_2$
	are vector fields along $\gamma$ of
	class $L^2$.
	As a consequence, setting
	\[
		\eta_3^i(t) =
		-
		\covI_{t_{i-1}}^t
		\left( \covI_{t_{i-1}}^\tau
			\Big(
				\covI_{t_{i-1}}^s
				h_3(\dot{\gamma},\covD_\sigma\dot{\gamma},
				\covD_\sigma^2\dot{\gamma})
				\mathrm{d}\sigma
			\Big)
		\mathrm{d}s \right)\mathrm{d}\tau,
		\qquad \forall t \in [t_{i-1},t_i],
	\]
	and
	\[
		\eta_2^i(t) =
		\covI_{t_{i-1}}^t
		\left( \covI_{t_{i-1}}^\tau
			w(\dot{\gamma},\covD_s^2\dot{\gamma})
		\mathrm{d}s \right)\mathrm{d}\tau,
		\qquad \forall t \in [t_{i-1},t_i],
	\]
	by integrating by parts we obtain
	\[
		\mathrm{d}f(\gamma)[\xi]
		= \int_{t_{i-1}}^{t_i}
		g\big(\covD_t^2\dot{\gamma}(t) - \eta_3^i(t) - \eta_2^i(t),
		\covD_t^3\xi(t)\big)
		\mathrm{d}t
		= 0.
	\]
	As a consequence, by Lemma~\ref{lem:DuBois-k-Riemannian}
	we obtain the existence of three parallel vector fields
	along $\gamma|_{[t_{i-1},t_i]}$,
	say $\zeta_0^i,\zeta_1^i$, and $\zeta_2^i$,
	such that
	\[
		\covD_t^2\dot{\gamma}(t)
		= \eta_3^i(t) + \eta_2^i(t)
		+ t^2 \zeta_2^i(t) + t\zeta_1^i(t) + \zeta_0^i(t),
		\qquad \forall t \in [t_{i-1},t_i].
	\]
	At this step, the bootstrap iterative procedure can start:
	since $\eta_2^i$ is of class $H^2$
	and $\eta_3^i$ is of class $H^3$,
	we obtain that $\covD_t^2\gamma$ is of class $H^2$,
	so $\gamma$ is of class $H^4$ in $[t_{i-1},t_i]$;
	but then, $\eta_2^i$ is of class $H^3$,
	and so $\gamma^i$ is of class $H^5$,
	and so forth.
	By this iteration, we have that $\gamma$
	is actually smooth in $[t_{i-1},t_i]$.

	For higher-order derivatives, a similar procedure will lead
	to the desired regularity results:
	let $k \ge 2$;
	for every $\ell = 0, \dots, k-2$,
	we denote by $\Omega_\ell \subset \mathbb{N}^4$
	the set of (ordered) quadruples of non-negative
	integers whose sum is equal to $\ell$.
	Hence, $\omega \in \Omega_\ell$
	if $\omega = (\omega^1, \omega^2, \omega^3, \omega^4)
	\in \mathbb{N}^4$
	and $\omega^1 + \omega^2 + \omega^3 + \omega^4 = \ell$.
	With this notation, for any $\xi \in V_{k,\gamma}$
	with compact support in $(t_{i-1},t_i)$,
	from~\eqref{eq:firstVariation-Riemannian-k} we 
	infer the following equation
	\begin{equation*}
		\begin{aligned}
			\mathrm{d}f_k(\gamma)[\xi]
			&
			= 
		\int_{t_{i-1}}^{t_i}
		g\Big(\covD_t^{k-1}\dot{\gamma},
			\covD_t^{k} \xi
			+
			\sum_{\ell = 0}^{k-2}
			\covD_t^\ell
			\big(R(\xi,\dot{\gamma})
				\covD_t^{k - \ell -2}
				\dot{\gamma}
			\big)
		\Big)\mathrm{d}t
		\\
			& = 
		\int_{t_{i-1}}^{t_i}
		g\big(\covD_t^{k-1}\dot{\gamma},\covD_t^k\xi\big)
		\mathrm{d}t
		+
		\int_{t_{i-1}}^{t_i}
		g\Big(\covD_t^{k-1}\dot{\gamma},
			\sum_{\ell = 0}^{k-2}
			\sum_{\omega \in \Omega_\ell}
			\big(
				(\covD_t^{\omega^1}R)
				(\covD_t^{\omega^2}\xi,\covD_t^{\omega^3}\dot{\gamma})
				\covD_t^{k - \ell - 2 + \omega^4}\dot{\gamma}
		\big)
		\Big)\mathrm{d}t.
		\end{aligned}
	\end{equation*}
	By the linearity of the integral operator
	and the symmetric and skew-symmetric properties of the 
	curvature tensor $R$, we then obtain
	\begin{equation*}
		\begin{aligned}
			\mathrm{d}f_k(\gamma)[\xi]
			& = 
		\int_{t_{i-1}}^{t_i}
		g\big(\covD_t^{k-1}\dot{\gamma},\covD_t^k\xi\big)
		\mathrm{d}t
		+
			\sum_{\ell = 0}^{k-2}
			\sum_{\omega \in \Omega_\ell}
		\int_{t_{i-1}}^{t_i}
		g\Big(\covD_t^{k-1}\dot{\gamma},
				(\covD_t^{\omega^1}R)
				(\covD_t^{\omega^2}\xi,\covD_t^{\omega^3}\dot{\gamma})
				\covD_t^{k - \ell - 2 + \omega^4}\dot{\gamma}
		\Big)\mathrm{d}t\\
			& = 
		\int_{t_{i-1}}^{t_i}
		g\big(\covD_t^{k-1}\dot{\gamma},\covD_t^k\xi\big)
		\mathrm{d}t
		-
			\sum_{\ell = 0}^{k-2}
			\sum_{\omega \in \Omega_\ell}
		\int_{t_{i-1}}^{t_i}
		g\Big((\covD_t^{\omega^1}R)
			(\covD_t^{k - \ell - 2 + \omega^4}\dot{\gamma},
			\covD_t^{k-1}\dot{\gamma})
			\covD_t^{\omega^3}\dot{\gamma},
			\covD_t^{\omega^2}\xi
		\Big)\mathrm{d}t.
		\end{aligned}
	\end{equation*}
	As a consequence, by rearranging the terms 
	we obtain $k-1$ vector fields along $\gamma$,
	say $h_2,h_3,\dots, h_k$,
	each depending on $\dot{\gamma}$
	and its covariant derivatives up to order $k-1$ 
	and, consequently, they are of class $L^2$,
	such that
	\begin{equation*}
		\mathrm{d}f_k(\gamma)[\xi]
		= 
		\int_{t_{i-1}}^{t_i}
		g\big(\covD_t^{k-1}\dot{\gamma},\covD_t^k\xi\big)
		\mathrm{d}t
		+
		\sum_{\ell = 2}^k
		\int_{t_{i-1}}^{t_i}
		g\big(h_\ell, \covD_t^{k-\ell}\xi\big)
		\mathrm{d}t.
	\end{equation*}
	Now, for each $\ell = 2,\dots,k$
	we apply the integration by part formula $\ell$
	times, obtaining a vector field 
	$\eta^i_\ell$ of class $H^{\ell}$,
	that with opportune changes of signs lead 
	to the following equation
	\begin{equation*}
		\begin{aligned}
		\mathrm{d}f_k(\gamma)[\xi]
		&
		= 
		\int_{t_{i-1}}^{t_i}
		g\big(\covD_t^{k-1}\dot{\gamma},\covD_t^k\xi\big)
		\mathrm{d}t
		-
		\sum_{\ell = 2}^k
		\int_{t_{i-1}}^{t_i}
		g\big(\eta^i_\ell, \covD_t^{k}\xi\big)
		\mathrm{d}t
		\\
		&=
		\int_{t_{i-1}}^{t_i}
		g\Big(\covD_t^{k-1}\dot{\gamma}
		-\sum_{\ell = 2}^k \eta^i_\ell,\covD_t^k\xi\Big)
		 = 0.
		\end{aligned}
	\end{equation*}
	Now, the bootstrap iterative procedure can start:
	by Lemma~\ref{lem:DuBois-k-Riemannian}
	we obtain $k$ parallel vector fields along $\gamma$,
	say $\zeta^i_0, \dots, \zeta_{k-1}^i$,
	such that 
	\[
		\covD_t^{k-1}\dot{\gamma}(t)
		= \sum_{\ell = 2}^k
		\eta^i_{\ell}(t)
		+ \sum_{l = 0}^{k-1}
		t^{l} \zeta^i_l(t),
		\qquad \forall t \in [t_{i-1}, t_i].
	\]
	As a consequence,
	$\covD_t^{k-1}\dot{\gamma}$
	has the same regularity as the least regular
	vector field on the right-hand side
	of the previous equation,
	namely $\eta^i_2 \in H^2([t_{i-1}, t_i], TM)$.
	But then $\eta^i_2$ is of class $H^4$
	and so on.
	Therefore, by this iterative procedure we conclude 
	that $\gamma|_{[t_{i-1}, t_i]}$ is a smooth function,
	for every $i = 1, \dots, N$.

	From this regularity,
	one can integrate by parts the first variation formula
	~\eqref{eq:firstVariation-Riemannian-k}
	and then apply the fundamental lemma of the calculus of variations
	to obtain~\eqref{eq:Riemannian-k-spline}.

	Finally, the $C^{2k-2}$ regularity of 
	$\gamma|_{[0,t_j]}$ and $\gamma|_{[t_j,1]}$
	and the conditions at the extrema 
	can be obtained
	from iterative integration by parts of
	~\eqref{eq:firstVariation-Riemannian-k},
	when applied to a general $\xi \in V_{k,\gamma}$
	and set to zero by the criticality condition of $\gamma$,
	and by~\eqref{eq:Riemannian-k-spline},
	using a similar procedure employed in the proofs
	of Theorem~\ref{theorem:regularity}
	and Proposition~\ref{prop:regularity-k}.
\end{proof}

\section{Existence of minimizers}
\label{sec:existence}

As stated in the introduction, the existence result 
we provide stems from a direct inspection 
of the counter-example provided as
a proof of~\cite[Lemma 2.15]{Heeren2019},
which states that a minimizer of $f$ can not exist
if only the interpolation points are prescribed.
For the reader's convenience, and because of its importance
in the subsequent discussion,
we report here the above counterexample.
\begin{example}
	\label{ex:cylinderExample}
	Let $M \subset \mathbb{R}^{3}$ be
	the following $2$--dimensional cylinder,
	endowed with the standard
	Euclidean metric induced by $\mathbb{R}^3$:
	\[
		M = \left\{
			(x,y,z)\in \mathbb{R}^{3}:
			x^2 + y^2 = \frac{1}{4\pi^2}
		\right\},
	\]
	so that the ``perimeter'' of $M$ has length equal to one.
	Let us set $(t_0,t_1,t_2) = (0,r,1)$,
	where $r \in (0,1)$ is an \emph{irrational} number,
	and let $p_0 = p_2 = \left(\frac{1}{2\pi},0,0\right)$,
	$p_1 = \left(-\frac{1}{2\pi},0,0\right)$.
	This setting is illustrated in Figure~\ref{fig:example}.
	Now, let us prove that 
	a minimizer of the spline energy functional
	among the smooth curves passing through that points 
	at these times doesn't exist.
	This can be done by showing the existence of a 
	minimizing sequence, i.e., 
	a sequence whose functional converges to its infimum,
	but the infimum cannot be achieved.
	\begin{figure}[ht]
		\centering
		\begin{minipage}{0.5\textwidth}
			\centering
			\begin{picture}(210,260)
				\put(0,0){\includegraphics[width=0.93\textwidth]{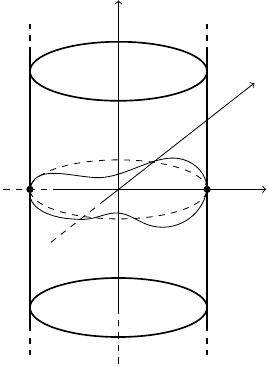}}
				\put(200,125){$x$}
				\put(193,203){$y$}
				\put(95,260){$z$}
				\put(160,140){$p_0 = p_2$}
				\put(7,140){$p_1$}
				\put(120,95){$\gamma$}
				\put(50,155){$\tilde\gamma$}
				\put(5,190){$M$}
			\end{picture}
		\end{minipage}%
		\begin{minipage}{0.5\textwidth}
			\centering
			\begin{picture}(210,260)
				\put(30,0){\includegraphics[width=0.78\textwidth]{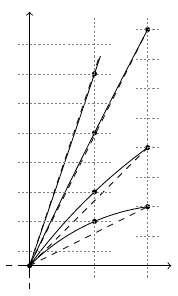}}
				\put(188,15){$t$}
				\put(122,17){$r$}
				\put(173,88){$(1,2)$}
				\put(173,146){$(1,4)$}
				\put(173,237){$(1,k)$}
				\put(26,65){$3/2$}
				\put(26,95){$5/2$}
				\put(12,204){$m + \frac{1}{2}$}
			\end{picture}
		\end{minipage}
		\caption{The setting of
			Example~\ref{ex:cylinderExample} (left)
			and its representation in the plane (right):
			On the plane, it can be observed how the
			straight lines (dashed in the figure)
			joining $(0,0)$ with points $(0,\mathbb{Z})$ can pass
			arbitrarily close to points
			in $(r, \frac{1}{2} + \mathbb{Z})$,
			thus requiring small perturbations
			to pass through those points,
			which implies a small increase
			in the spline functional.
		}
		\label{fig:example}
	\end{figure}
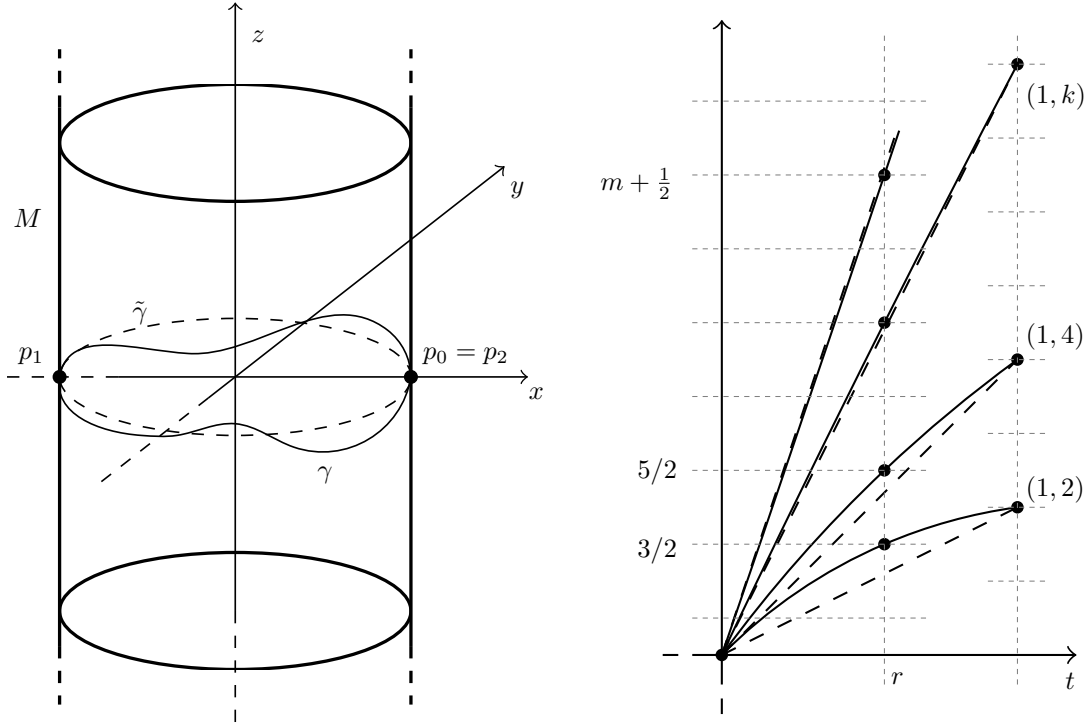

	First of all, let us notice that for any curve on $M$,
	its projection on the plane $z \equiv 0$
	has lower spline energy.
	More formally, for any curve
	$\gamma\colon [0,1] \to M$,
	$\gamma(t) = (\gamma_1(t),\gamma_2(t),\gamma_3(t))$,
	setting 
	$\tilde\gamma(t) = (\gamma_1(t),\gamma_2(t),0)$,
	we have 
	$f(\tilde\gamma) \le f(\gamma)$.
	So, in our interpolation problem,
	we can restrict our analysis on the 
	curves that lie on the $z \equiv 0$ plane
	(see Figure~\ref{fig:example}).

	Secondly, using the standard identification
	of $S^1$ with $\mathbb{R}/\mathbb{Z}$,
	the problem can be cast in the following setting:
	\[
		\inf \left\{
			\int_0^1 \ddot{\gamma}^2(t)\mathrm{d}t:
			\gamma\in H^2([0,1],\mathbb{R}),
			\gamma(0) = 0,
			\gamma(r) \in \frac{1}{2} + \mathbb{Z},
			\gamma(1) \in \mathbb{Z}
		\right\},
	\]
	as represented in Figure~\ref{fig:example}.
	By the irrationality of $r$, any straight line joining the 
	point $(0,0)$ with a point in  $(0,\mathbb{Z})$
	doesn't pass through $(r,\frac{1}{2} + \mathbb{Z})$.
	Hence, the spline energy is strictly greater than zero
	for any curve that satisfies the prescribed 
	interpolation conditions,
	meaning that the infimum of the above problem 
	is greater then or equal to $0$,
	and that the minimum, if it exists, can't be zero.
	However, if we consider the sequence of parabolas
	$(\gamma_{k})_{k \in \mathbb{N}}\colon [0,1] \to \mathbb{R}$
	such that $\gamma_k(0) = 0$, $\gamma_k(1) = k$
	and $\gamma_k(r) = m + 1/2$,
	where $m \in \mathbb{Z}$ will be determined later,
	a direct computation shows that
	\[
		\int_{0}^{1}\ddot{\gamma}_k^2(t)\mathrm{d}t
		=  4\frac{\left(m + \frac{1}{2} - kr\right)^2}{(r^2 - r)^2}.
	\]
	By Dirichlet's approximating theorem,
	there exists a subsequence 
	$(\gamma_{k_j})_{j \in \mathbb{N}}$
	and a sequence $m_j \in \mathbb{N}$
	such that
	\[
		\lim_{j \to \infty} \Big(m_j +\frac{1}{2} - k_{j}r\Big) = 0,
	\]
	see Figure~\ref{fig:example}.
	Hence, we have that $\lim_{j \to \infty}f(\gamma_{k_j}) = 0$,
	so that $(\gamma_{k_j})_{j \in \mathbb{N}}$ is a minimizing sequence
	and the infimum is indeed $0$, which can't be a minimum.
	Therefore, a minimizer for the spline energy functional 
	with the above interpolation constraints does not exist.
\end{example}
\begin{remark}
	\label{rem:diverging-velocities}
	As one can infer by a direct inspection
	of Example~\ref{ex:cylinderExample},
	the construction of the minimizing sequence is obtained 
	by a small perturbation of straight lines that are 
	approximating the interpolation conditions,
	and such a sequence of maps can exist only if
	no velocity is prescribed,
	since the initial velocity of such a sequence of functions
	has to diverge.
	Indeed, we have
	\[
		\lim_{j \to \infty}|\dot{\gamma}_{k_j}| = 
		\lim_{j \to \infty}\left(
			k_j - \frac{2m_j + 1 - 2k_jr}{r^2 - r}
		\right)
		= +\infty.
	\]
\end{remark}
Remark~\ref{rem:diverging-velocities} suggests that,
if at least one of the velocities of the curves is prescribed,
for example the initial one, then 
in Example~\ref{ex:cylinderExample}
the set of points $(r,\mathbb{Z} + 1/2)$
that can be achieved without increasing
``too much'' the spline energy functional is finite.
This is the heuristically observation that suggested
the existence result
formally given by Theorem~\ref{theorem:existence}.

\begin{remark}
	One can be easily convinced that if
	natural or periodic boundary conditions are prescribed, 
	then a minimizer of $f$ can not exist when $\sigma = 0$.
	Considering once again the problem
	given in Example~\ref{ex:cylinderExample},
	this can be understood by observing
	that suitable small perturbations of the straight 
	lines passing through $(0,0)$ and $(1,k)$
	lead to a sequence of curves $(\gamma_{k})_{k \in \mathbb{N}}$
	satisfying the interpolation conditions and
	both the natural and periodic boundary conditions,
	all while keeping the increase
	in the spline energy functional relatively low.
	Consequently, we still have
	$\lim_{k \to \infty} f(\gamma_k) = 0$,
	indicating the absence of a minimizer.

	For the sake of precision,
	let us explicitly construct one of these sequences.
	Let $(\omega_{k_j})_{j \in \mathbb{N}}$ be the 
	sequence of straight lines passing through
	$(0,0)$ and $(1,k_j)$;
	as previously done, let us also consider a sequence 
	$(m_j)_{j \in \mathbb{N}}\subset N$ such that
	\[
		\alpha_j \coloneqq m_j + \frac{1}{2} -k_j r \to 0,
		\qquad \text{as } j \to \infty.
	\]
	As a consequence, 
	the distance between the point $(r,m_{j} + \frac{1}{2})$
	and the line $\omega_{k_j}$ is converging to zero.
	Let us choose $\delta > 0$ such that 
	$[r-\delta,r+\delta] \subset (0,1)$
	and a function $\phi \in C^{2}([0,1],\mathbb{R})$
	such that $\supp \phi \subset (r-\delta,r+\delta)$,
	and $\phi(r) = 1$.
	Then, every function $\gamma_{j}\in C^2([0,1],\mathbb{R})$
	defined by $\gamma_j = \omega_{k_j} + \alpha_j\phi$
	satisfies both the natural and the periodic
	boundary conditions, since
	$\ddot{\gamma}_j(0) = \ddot{\gamma}_j(1) = 0$
	and 
	$\dot{\gamma}_j(0) = \dot{\gamma}_j(1) = k_j$,
	respectively, and it satisfies the interpolation 
	condition at $r$ by construction.
	Moreover,
	\[
		\lim_{j \to \infty}f(\gamma_j)
		= \lim_{j \to \infty}
		\alpha_j^2
		\int_{0}^{1} \ddot{\phi}^2(t)\mathrm{d}t
		= 0,
	\]
	hence $(\gamma_j)_{j \in \mathbb{N}}$ is a minimizing sequence,
	but the problem doesn't admit a minimum.
\end{remark}

\begin{remark}
	Before giving the proof of
	Theorem~\ref{theorem:existence},
	let us make an important observation 
	about the number of constraints 
	and the number of parameters we have.
	We give it having in mind the one dimensional setting,
	but the same arguments hold for
	the more general Riemannian setting.
	From the regularity result
	(i.e.,Theorem~\ref{theorem:regularity}),
	we know that a minimizer is 
	a cubic spline that, in each interval $[t_{i-1},t_{i}]$
	can be uniquely determined by $4$ parameters.
	Therefore, each cubic spline is
	characterized by $4N$ parameters. 

	On the other hand, the constraints
	we impose on the problem give a total 
	of $4N$ constraints on the parameters 
	of the cubic splines,
	independently on where the velocity is imposed.
	Let us explicitly count them in the case
	of prescribed initial velocity.
	Since the curve has to pass through 
	the points $p_0,\ldots,p_N \in M$,
	each cubic polynomial is constraints on the extreme points,
	and this provides $2N$ constraints.
	Moreover, the first one has prescribed 
	velocity, while the last one,
	because of the natural boundary condition
	$\ddot{\gamma}(1) = 0$,
	has another constraint,
	so that we have $2N + 2$ constrained parameters.
	Moreover, by Theorem~\ref{theorem:regularity},
	the minimizer is a function of class $C^2$,
	hence on each of the points $p_1,\ldots,p_{N-1}$
	we must impose other $2$ constraints,
	i.e., $\dot{\gamma}(t_i^-) = \dot{\gamma}(t_i^+)$
	and  $\ddot{\gamma}(t_i^-) = \ddot{\gamma}(t_i^+)$,
	and this regularity condition provides
	$2(N-1)$ constraints.
	Summing up, the total number of the constraints 
	is $2N + 2 + 2(N-1) = 4N$,
	so the minimizer is uniquely determined.

	If the velocity is prescribed at one 
	interior time $t_j \in (0,1)$,
	then two natural boundary 
	conditions have to be taken into account,
	so that the number of constraints at the extrema
	doesn't change;
	because both $\dot{\gamma}(t_j^-)$
	and $\dot{\gamma}(t_j^+)$ must be equal 
	to $v$,
	we have an additional constraint,
	but at the same time we lose one constraint 
	on the regularity, since in $t_j$
	the minimizer has only the $C^1$--regularity.
	Therefore, even in this case, 
	we have a total of $4N$ constraints,
	and the minimizer is unique.
\end{remark}

\subsection{Proof of Theorem~\ref{theorem:existence}}
The proof of Theorem~\ref{theorem:existence}
follows the direct method of the calculus of variations.
Specifically, we will prove the coercivity of the functional
$f\colon \Gamma \to \mathbb{R}$,
which ensures the weak convergence of a minimizing sequence
(more formally, of one of its subsequences).
Additionally, we will demonstrate that the functional is
weakly lower semicontinuous,
meaning that if $\gamma_k$ weakly converges 
to $\tilde\gamma$,
then $\liminf_{k \to \infty} f(\gamma_k) \ge f(\tilde{\gamma})$,
so that $\tilde{\gamma}$ is a minimizer.
These results have already been obtained
in~\cite[Section 3]{GGP2002} for the case
of initial and final prescribed points and velocities,
but a close inspection of the arguments shows
that the same procedure applies in our setting.
For the reader's convenience, and for a self-contained work,
we provide the proof here with the necessary modifications.

Let $\Gamma \subset H^2([0,1],M)$
be given as in Theorem~\ref{theorem:existence}.
We start with the following coercivity result,
whose proof relies on the boundary condition 
on the velocity at the initial time,
i.e.,
$\dot{\gamma}(0) = v$ for every $\gamma \in \Gamma$.
\begin{lemma}
	\label{lem:coercivity}
	Let $(\gamma_k)_{k \in N}\subset \Gamma$ be such that 
	$\sup_{k \in N}f(\gamma_k) < +\infty$.
	Then, $g(\dot{\gamma}_k,\dot{\gamma}_k)$ is uniformly bounded
	and $(\gamma_k)_{k \in N}$ admits a subsequence that 
	weakly converges, with respect to the $H^2$--norm,
	to a curve in $\Gamma$.
\end{lemma}

\begin{proof}
	The proof follows along the lines of~\cite[Lemma 3.1]{GGP2002}.
	Let $c^2 = \sup_{k \in \mathbb{N}} f(\gamma_k)$
	and define, for every $k$,
	$H_k\colon [0,1] \to \mathbb{R}$
	as $H_k(t) \coloneqq g(\dot{\gamma}_k(t), \dot{\gamma}_k(t))$.
	Since $\dot{\gamma}_k(0) = v$ for every $k$,
	setting $h_0 = g(v,v)$,
	we obtain the following estimate,
	which also uses the classical H\"older inequality:
	\begin{multline*}
		H_k(t) 
		= h_0 + \int_0^t \frac{\mathrm{d}}{\mathrm{d}\tau} H_k(\tau) \mathrm{d}\tau
		\le h_0 + 2 \int_0^1 \Big|
		g\big(\covD_\tau \dot{\gamma}_k, \dot{\gamma}_k\big)
		\Big| \mathrm{d}\tau\\
		\le h_0 + 2 \big(f(\gamma_k)\big)^{\frac{1}{2}}
		\Big(\int_0^1
			g(\dot{\gamma}_k, \dot{\gamma}_k) \mathrm{d}\tau
		\Big)^{\frac{1}{2}}
		\le h_0 + 2c \norm{H_k}_{L^\infty}^{\frac{1}{2}},
		\qquad \forall t \in [0,1].
	\end{multline*}
	By the above inequality, taking the supremum over $t \in [0,1]$,
	we obtain that $\norm{H_k}_{L^\infty}$ is uniformly bounded.

	Let us show the existence of a subsequence that 
	weakly converges to a curve in $\Gamma$.
	As a first step, let us notice that the uniform boundness 
	of $\norm{H_k}_{L^\infty}$ implies
	the weak convergence with respect to the 
	$H^1$--norm of a subsequence of  $(\gamma_k)_{k \in \mathbb{N}}$,
	which we denote again by $(\gamma_k)_{k \in \mathbb{N}}$,
	to a curve $\tilde\gamma$.
	Moreover, by the Ascoli-Arzel\`a Theorem,
	this convergence is also uniform.

	Now, using a finite cover of $\tilde\gamma$,
	we will prove the weak $H^2$--convergence.
	Let $I_{1},I_2,\dots,I_L \subset [0,1]$
	a finite collection of closed interval
	such that 
	$[0,1] = \cup_{\ell = 1}^L I_\ell$,
	and each $\tilde\gamma|_{I_\ell}$
	is inside a (open) local chart $(U_\ell,\varphi_\ell)$
	for $k = 1,\dots,K$.
	More precisely, this means that, for every $k$,
	$U_\ell$ is an open connected subset of $\mathbb{R}^n$,
	$\varphi_\ell\colon U_\ell \to M$ is a (smooth) diffeomorphism
	and $\tilde\gamma(I_\ell)\subset \varphi_\ell(U_\ell)$.
	Since $\gamma_n$ is uniformly convergent to $\tilde\gamma$,
	without loss of generality we can assume that also
	$\gamma_k(I_\ell)\subset \varphi_\ell(U_\ell)$ for every 
	$\ell$ and for every $k$.
	Let us denote by $(x_k)_{k\in\mathbb{N}}$ and $\tilde{x}$
	the parametrizations of $(\gamma_k)_{k \in \mathbb{N}}$
	and $\tilde\gamma$, respectively,
	in these local charts, independently on $\ell$,
	namely $\gamma_k(t) = \varphi_\ell(x_k(t))$
	for every $k$ and for every $t \in I_\ell$,
	and similarly for $\tilde\gamma$.
	We will prove that, for every $\ell$,
	$(x_k)_{k \in\mathbb{N}}$ weakly converges
	in the $H^2$--norm to $\tilde{x}$,
	up to subsequences,
	and this will end the proof.
	Hence, it suffices to prove that 
	$\norm{\ddot{x}_k}_{L^2(I_\ell)}$
	is uniformly bounded, independently on $\ell$.

	By the uniform convergence of $(\gamma_k)_{k \in \mathbb{N}}$
	to $\tilde\gamma_k$,
	there exists a compact subset of $M$ that contains
	all the images of $\gamma_k$.
	This implies that, setting
	\[
		g_\ell(v,w) = g(\mathrm{d}\varphi_\ell(v),\mathrm{d}\varphi_\ell(w)),
		\qquad \forall v,w \in \mathbb{R}^n,
	\]
	we can assume the existence of two constants $\alpha,\beta > 0$
	such that
	\[
		\alpha\norm{\ddot{x}_k}^2
		\le g_\ell(\ddot{x}_k,\ddot{x}_k)
		\le \beta\norm{\ddot{x}_k}^2,
		\qquad \forall \ell = 1,\dots, L.
	\]
	As a consequence, we need to prove that 
	\[
		\int_{I_\ell} g_k(\ddot{x}_k,\ddot{x}_k)\mathrm{d}t
	\]
	is uniformly bounded, independently on $k$. 
	In every local chart, the 
	components of the $\covD_t\dot{\gamma}_k$
	are given by a formula of the following kind
	\[
		\ddot{x}_k + \Gamma_\ell(x_k,\dot{x}_k),
	\]
	where $\Gamma_\ell\colon U_\ell\times \mathbb{R}^n \to \mathbb{R}^n$
	is a continuous map homogeneous of degree $2$
	with respect to the second variable,
	which can be explicitly computed from the
	Christoffel symbols of the connection 
	in the local coordinates.
	By the homogeneity of degree $2$ of the maps 
	$\Gamma_\ell$, and by the uniform boundness of
	$g(\dot{\gamma}_k,\dot{\gamma_k})
	= g_\ell(\dot{x}_k,\dot{x}_k)$,
	we infer the existence of a constant $a > 0$
	such that
	\[
		\int_{I_\ell}
		g_\ell\big(\Gamma_\ell(x_k,\dot{x}_k,),\Gamma_\ell(x_k,\dot{x}_k)\big)
		\mathrm{d}t
		\le a^2,
		\qquad \forall \ell = 1,\dots,L.
	\]
	Since $f(\gamma_k) \le c^2$ for every $k \in \mathbb{N}$,
	we then obtain the following estimate
	\begin{multline*}
		c^2 \ge f(\gamma_k)
		\ge \frac{1}{2}
		\int_{I_\ell}g\big(\covD_t\dot{\gamma}_k,\covD_t\dot{\gamma}_k\big)
		\mathrm{d}t
		=\frac{1}{2}
		\int_{I_\ell}
		g_\ell\big(\ddot{x}_k + \Gamma_\ell(x_k,\dot{x}_k),
		\ddot{x}_k + \Gamma_\ell(x_k,\dot{x}_k)\big)
		\mathrm{d}t\\
		= \frac{1}{2}
		\int_{I_\ell}
		g_\ell\big(\ddot{x}_k,\ddot{x}_k\big)\mathrm{d}t
		+
		\int_{I_\ell}
		g_\ell\big(\ddot{x}_k,\Gamma_\ell(x_k,\dot{x}_k)\big)\mathrm{d}t
		+ \frac{1}{2}
		\int_{I_\ell}
		g_\ell\big(\Gamma_\ell(x_k,\dot{x}_k),\Gamma_\ell(x_k,\dot{x}_k)\big)\mathrm{d}t\\
		\ge \frac{1}{2}
		\int_{I_\ell}
		g_\ell\big(\ddot{x}_k,\ddot{x}_k\big)\mathrm{d}t
		-
		\left( \int_{I_\ell}
			g_\ell\big(\ddot{x}_k,\ddot{x}_k\big)\mathrm{d}t
		\right)^{\frac{1}{2}}
		\left( \int_{I_\ell}
			g_\ell\big(\Gamma_\ell(x_k,\dot{x}_k),\Gamma_\ell(x_k,\dot{x}_k)\big)\mathrm{d}t
		\right)^{\frac{1}{2}}\\
		\ge \frac{1}{2}
		\int_{I_\ell}
		g_\ell\big(\ddot{x}_k,\ddot{x}_k\big)\mathrm{d}t
		- a
		\left( \int_{I_\ell}
			g_\ell\big(\ddot{x}_k,\ddot{x}_k\big)\mathrm{d}t
		\right)^{\frac{1}{2}},
	\end{multline*}
	from we we infer
	\[
		\int_{I_\ell}
		g_\ell\big(\ddot{x}_k,\ddot{x}_k\big)\mathrm{d}t
		\le \big(a + \sqrt{a^2 + 2c^2}\big)^2.
	\]
	Since the right hand term is independent on $k \in \mathbb{N}$
	and $\ell \in 1,\dots,L$,
	we are done.
\end{proof}
\begin{proof}[Proof of Theorem~\ref{theorem:existence}]
	Since $f$ is lower bounded, 
	it admits a minimizing sequence
	$(\gamma_k)_{k \in N}\subset \Gamma$ such that 
	\[
		m \coloneqq
		\inf_{\gamma \in \Gamma} f(\gamma) =
		\lim_{k \to \infty}f(\gamma_k) \in \mathbb{R}.
	\]
	Therefore, without loss of generality, we can assume that 
	\[
		f(\gamma_k) \le m +1, \qquad \forall k \in \mathbb{N}.
	\]
	By Lemma~\ref{lem:coercivity},
	$(\gamma_{k})_{k \in \mathbb{N}}$
	weakly converges with respect to the $H^2$--norm
	to a curve $\tilde\gamma \in \Gamma$, up to subsequences.
	Therefore, it suffices to prove that $f$
	is weakly lower semicontinuous with respect to this norm,
	so that
	\begin{equation}
		\label{eq:f-wlsc}
		f(\tilde\gamma)\le \liminf_{k \to \infty}f(\gamma_k) = m,
	\end{equation}
	namely, $\tilde\gamma$ is a minimizer of $f$.

	Using again the localization argument
	employed in the proof of Lemma~\ref{lem:coercivity},
	together with its notation,
	let $(U_\ell,\varphi_\ell)_{\ell = 1,\dots,L}$
	a collection of local charts 
	that covers $\tilde\gamma([0,1])$.
	Taking a subsequence if necessary,
	we can also assume that 
	$(U_\ell,\varphi_\ell)_{\ell = 1,\dots,L}$
	covers also the image of each $\gamma_k$.
	Then,~\eqref{eq:f-wlsc} is obtained by proving 
	that 
	\begin{equation}
		\label{eq:f-wlsc-local}
		\int_{I_\ell}
		g\big(\covD_t\dot{\tilde\gamma},\covD_t\dot{\tilde\gamma}\big)
		\mathrm{d}t
		\le \liminf_{k \to \infty}
		\int_{I_\ell} g\big(\covD_t\dot\gamma_k,\covD_t\dot\gamma_k\big)
		\mathrm{d}t,
		\qquad \forall \ell = 1,\dots,L.
	\end{equation}
	Since $(x_k)_{k \in \mathbb{N}}$
	$x_k \rightharpoonup \tilde{x}$
	in $H^2(I_\ell,U_\ell)$,
	it converges to $x$ with respect to $C^1$--norm
	and, as a consequence,
	$\Gamma_\ell(x_k,\dot{x}_k)$
	converges to $\Gamma_\ell(\tilde{x},\dot{\tilde{x}})$
	with respect to the $L^{\infty}$ norm,
	so strongly with respect to $L^2$, meaning that
	\[
		\lim_{k \to\infty}
		\int_{I_l}
		g_\ell\big(\Gamma_\ell(x_k,\dot{x}_k),
		\Gamma_\ell(x_k,\dot{x}_k)\big)
		\mathrm{d}t
		= \int_{I_l}
		g_\ell\big(\Gamma_\ell(\tilde{x},\dot{\tilde{x}}),
		\Gamma_\ell(\tilde{x},\dot{\tilde{x}})\big)
		\mathrm{d}t,
		\qquad \forall \ell = 1,\ldots,L.
	\]
	Moreover, thanks to the weak convergence of
	$\ddot{x}_k$ to $\ddot{\tilde{x}}$ in the $L^2$--norm,
	we have also that
	\[
		\lim_{k \to \infty}
		\int_{I_\ell}
		g_\ell\big(\ddot{x}_k,\Gamma_\ell(x_k,\dot{x}_k)\big)
		\mathrm{d}t
		=
		\lim_{k \to \infty}
		\int_{I_\ell}
		g_\ell\big(\ddot{x}_k,\Gamma_\ell(\tilde{x},\dot{\tilde{x}})\big)
		\mathrm{d}t
		=
		\int_{I_\ell}
		g_\ell\big(\ddot{\tilde{x}},\Gamma_\ell(\tilde{x},\dot{\tilde{x}})\big)
		\mathrm{d}t,
	\]
	for every $\ell = 1,\dots,L$.
	Summing up,~\eqref{eq:f-wlsc-local} 
	reduces as follows:
	\[
		\int_{I_\ell}
		g_\ell\big(\ddot{\tilde{x}},\ddot{\tilde{x}}\big) \mathrm{d}t
		\le \liminf_{k \to \infty} \int_{I_\ell}
		g_\ell\big(\ddot{x}_k,\ddot{x}_k\big) \mathrm{d}t,
		\qquad \forall \ell = 1,\dots, L,
	\]
	which is nothing but the weak lower semicontinuity of 
	the norm in $L^2$,
	hence we are done.
\end{proof}

\bigskip
\noindent
\textbf{Acknowledgements} 

\noindent
D. Corona thanks the partial support of INdAM 
(Italian National Institute of High Mathematics),
Grant: ``mensilit\`a di borse di studio per l’estero per l’a.a. 2023-2024''.

\noindent
D. Corona and R. Giamb\`o thank the partial support of GNAMPA INdAM
(Italian National Institute of High Mathematics),
Project: CUP-E53C22001930001.

\noindent
D. Corona and P. Piccione thank the partial support of FAPESP
(Projeto Temático Fapesp 2022/16097-2).

\bigskip
\begin{tabular}{lp{1cm}l}
	\textbf{D. Corona} & & 
	\textbf{R. Giamb\`o} \\
	Universit\`a degli Studi di Camerino & &
	Universit\`a degli Studi di Camerino\\
	School of Science and Technology  & & 
	School of Science and Technology \\
	Via Madonna delle Carceri 9 & & 
	Via Madonna delle Carceri 9 \\
	62032	-- Camerino (MC), Italy & & 
	62032	-- Camerino (MC), Italy \\
	\emph{E-mail}: {\tt dario.corona@unicam.it} & & 
	\emph{E-mail}: {\tt roberto.giambo@unicam.it}\\
	[0.5cm]
	\textbf{P. Piccione} & & \\
	Zhejiang Normal University, & & \\
	School of Mathematical Sciences, & & \\
	321004, Jinhua-ZJ, & & \\
	People’s Republic of China. & & \\
	\emph{Permanent Address}: & &\\
	Universidade de S\~ao Paulo & & \\
	Departamento de Matem\'atica & & \\
	Rua do Mat\~ao 1010 & & \\
	S\~ao Paulo, SP 05508--090, Brazil & & \\
	\emph{E-mail}: {\tt paolo.piccione@usp.br}
\end{tabular}
\end{document}